\documentclass[11pt,a4paper]{article}
\usepackage{hyperref}
\usepackage{amsmath,amssymb,amsthm,dsfont,enumerate,mathrsfs}
\usepackage{graphicx}               
\usepackage[english]{babel}
\usepackage{dsfont}
\usepackage{subfig}
\usepackage{float}
\usepackage{microtype}
\usepackage[textsize=small]{todonotes}       

\usepackage{titleps}
\newpagestyle{classica}{%
	\sethead{}{Curie-Weiss Widom-Rowlinson model under time-evolution}{\thepage}
}

\newtheorem{theorem}{Theorem}[section]

\newtheorem{definition}[theorem]{Definition}
\newtheorem{lemma}[theorem]{Lemma}

\newtheorem{proposition}[theorem]{Proposition}

\newcommand{\be}{\begin{equation}}
\newcommand{\ee}{\end{equation}}

\newcommand{\al}{\alpha}

\newcommand{\arctanh}{{\rm atanh}}

\numberwithin{equation}{section}

\DeclareMathOperator*{\argmax}{arg\,max}

\def\e{\varepsilon}
\def\a{\alpha}
\def\b{\beta}
\def\g{\gamma}

\def\d{\delta}

\def\O{\Omega}
\def\o{\omega}

\def\R{\mathbb{R}}

\def\E{\mathbb E}

\title{Dynamical Gibbs-non-Gibbs transitions in Curie-Weiss 
Widom-Rowlinson models }

\author
{
 Sascha Kissel\footnote{
   Ruhr-Universit\"at Bochum, Fakult\"at f\"ur Mathematik, Universit\"atsstra\ss e 150, 44780 Bochum, Germany.
   E-mail: {\tt sascha.kissel@ruhr-uni-bochum.de},  {\tt christof.kuelske@ruhr-uni-bochum.de}
 }
 \and Christof K\"ulske\footnotemark[\value{footnote}] 
 }
\date{\today}

\begin{document} 
\maketitle 

\begin{abstract} We consider the Curie-Weiss Widom-Rowlinson model for particles with spins and holes, 
with a repulsion strength $\b>0$ 
between particles of opposite spins.  We provide a closed solution of the model, and investigate 
dynamical Gibbs-non-Gibbs transitions for the time-evolved model under independent 
stochastic symmetric spin-flip dynamics.  
We show that, for sufficiently large $\b$ after a transition time, 
continuously many bad empirical measures appear. These lie on (unions of) curves on the simplex  
whose time-evolution we describe.  
%
%
%
%
%
%

\end{abstract}
\vspace{0.3cm} 
\newpage
\section{Introduction}

The investigation of dynamical Gibbs-non Gibbs transitions can be undertaken 
for models in different geometries, in particular for lattice systems, for mean-field systems, for Kac-systems, 
for systems of point particles in the continuum.  

Historically the first example of such a study of the loss 
and possible recovery of the Gibbs property in the course of a time evolution from an initial 
infinite-volume Gibbs measure 
was given for the Ising model on the lattice, under independent symmetric spin-flip, cf. \cite{enter-fernandez-hollander-redig02}. 
The Curie-Weiss Ising model under symmetric spin-flip was first investigated 
in \cite{kuelske-le07}, using the appropriate notion of sequential Gibbsianness (see below), 
see also \cite{fernandaz-hollander-martinez13}. The notion of sequential Gibbsianness is to be used for Kac-models 
on the torus, too, for which spin configurations have a spatial structure, 
where it relates to hydrodynamic 
scaling, cf. \cite{fernandaz-hollander-martinez14},\cite{jahnel-kuelske17a},\cite{henning-kraaij-kuelske}. 
In the time-evolved Curie-Weiss Ising model non-Gibbsian behavior at low temperatures appears with 
symmetry-breaking in the set of bad magnetizations for an intermediate time-interval, 
and this happens already under independent spin-flip. 
A variety of interesting phenomena appear 
for interacting dynamics,  
in particular in the regime of strongly interacting dynamics, which gives rise to periodic orbits in 
the associated Hamiltonian flow, cf. \cite{ermolaev-kuelske10},\cite{kraaij-redig-zuijlen17}. 

For systems of point particles in infinite Euclidean space, the Gibbsian formalism is well-established (see \cite{preston-book}\cite{ruelle-book},\cite{dereudre}) and statements which are analogous to those for lattice systems tend to be more difficult. 
An important such system is the Widom-Rowlinson model.
It has a repulsive interaction between particles of different colors, 
and shows a phase-transition at high intensity, proved by 
Peierls arguments or percolation ideas, cf. \cite{chayes-chayes-kotecky95},\cite{ruelle71},\cite{bricmont-kuroda-lebowitz85}. 

In \cite{jahnel-kuelske17b} dynamical GnG transitions for the WR model in Euclidean space 
with hardcore intercolor interaction 
were investigated under independent spin-flip dynamics 
which keeps the spatial degrees of freedom fixed. 
The main features found in that analysis were an immediate loss of the Gibbs property and
the possibility of full-measure discontinuities for the time-evolved measure in the percolating region. 
Immediate loss is quite unusual in the lattice world for regular interactions  
(see results for the preservation of short-time Gibbsianness \cite{le-redig02},\cite{kuelske-opoku08}), 
and in the mean-field world (see however the somewhat pathological example 
of \cite{hollander-redig-zuijlen15}). Full-measure discontinuities under time-evolution 
had not been observed for lattice or mean-field systems so far, however they 
might appear on trees \cite{enter-ermolaev-Iacobelli-kuelske12},  
see also the examples of transformed measures not coming from a time-evolution 
showing full-measure discontinuities 
on the lattice in \cite{kuelske-le-redig04} and in mean field \cite{kuelske03}. 
Natural versions of the WR model are formulated 
also as a lattice system 
(\cite{gallavotti-lebowitz71},\cite{higuchi-takei04}), or as a mean-field system which we will 
study here.  
It is the purpose of this note to investigate the Curie-Weiss WR model with a soft repulsion with a strength 
$\beta>0$, under independent symmetric spin-flip dynamics, and give a detailed description of 
the types of transitions and their sets of bad empirical measures. 

In the first step we provide the necessary static analysis: 
The Curie-Weiss WR model is an extension of the Curie-Weiss Ising model
(which is recovered as a special case for full occupation density) with the additional degrees of freedom  
due to the occurrence of holes. 
Using suitable parametrizations, the model is solved in terms of closed solution formulas, see Theorem \ref{closedsolution}, relating typical empirical 
measures for spins and holes to model parameters $\b$, and the a priori distribution $\al$. 
It shares some properties  with the Curie-Weiss Ising model, but it is richer: 
Like the Ising model it has a second order phase transition in a magnetization variable, 
with usual mean-field critical exponents, unlike the Ising model it  
has a second order phase transition in occupation density in its attractive (antiferromagnetic) version. 
For related but different work in the grand-canonical framework, see \cite{georgii-zagrebnov98},\cite{kozitsky-kozlovskii}.


Next we come to the dynamics, for which we restrict to the symmetric model at time zero 
with equal a priori probabilities for plus particles and minus particles.  
We show the following: For small enough repulsion $\beta\leq 2$ the model preserves the sequentially Gibbs property for all times. For strong enough repulsion $\b>2$, the model loses the sequential Gibbs property after a finite time, 
and a continuum of bad empirical measures on the simplex appears which evolves with time and never becomes empty again. 
In the most interesting regime, at very strong repulsion $\beta>3$, the set of 
bad empirical measures undergoes the following type of time-evolution: starting from  the empty set for small 
times, two symmetric arcs appear at a transition time, from these    
a Y-shaped region is formed, which then ultimately degenerates at a final transition time into a growing line. 
Our analysis relies on conditional large deviations, 
where we are able to make use of previous results for the Curie-Weiss Ising model \cite{kuelske-le07}, 
for the relevant bifurcation analysis (with appearance of the Butterfly-singularity, see \cite{poston-stewart-book}). 
 Finally we discuss and illustrate the almost-Gibbsian behavior of the time-evolved model,  
 see fig. \ref{fig: atypic 2.8} and \ref{fig: atypic 5}.

\section{Model and main results}
 
\subsection{The Curie-Weiss Widom-Rowlinson model and sequential Gibbs property}

We denote the single-site state space by $E:= \{-1,0,1\}$.  
We write $\Omega_N = E^N$ for the state space at finite system size $N\in \mathbb{N}$.

\begin{definition}
	 The finite-volume Gibbs measure  at system size $N\in\mathbb{N}$ of the {\em Curie-Weiss 
	 	Widom-Rowlinson model}  with 
	 a priori measure $\al\in \mathcal{M}_1(E)$ and repulsion strength $\beta>0$ is defined to be the probability 
	 measure on $\Omega_N$ given by
	 \begin{align}\label{WR}
	 \mu_{N,\beta,\al}(\omega_{[1,N]}):=\frac{1}{Z_{N,\beta,\al}} e^{-\frac{\beta}{2N} \sum_{1\leq i,j\leq N}\mathds{1}(\omega_i\omega_j =-1)}\prod_{j=1}^N \al(\omega_j)
	 \end{align}
	 for $\omega_{[1,N]}= (\omega_i)_{1\leq i \leq N}\in E^N$ where is  the partition function $Z_{N,\beta,\al}$ is determined  by the normalization requirement. 
\end{definition}

%

If $\omega_i=0$ we say that there is no particle at site $i$, if $|\omega_i|=1$ we say that a particle is present at $i$, where  
we interpret the value $-1$ as particle with a negative spin, and $+1$ as a particle with positive spin. 
In the model there is no interaction between particles and holes, no interaction between particles of the same sign, 
but a repulsion between pairs of particles of opposite spin with strength $\b>0$.  The interaction disfavors configurations with many particles 
of opposite signs present, so it is of a ferromagnetic type.  

%
For a given a priori measure  we call 
$\a(\{1,-1\})$ the occupation density, $\a(0)$ the hole density and write $\al^*:=\frac{\al(1)-\al(-1)}{\al(1)+\al(-1)}$ for the magnetization on occupied sites. We call the a priori measure ($\pm$)-symmetric if $\al^*=0$.
For our study of time-evolved measures below we will use the intrinsic definition of sequential 
Gibbsianness for sequences of permutation invariant measures (See \cite{haeggstroem-kuelske04}). 
\begin{definition}\label{def: seqG}
	A sequence of exchangeable measures  $\mu_N\in \mathcal{M}_1(\O_N)$ 
	is called sequentially Gibbs iff for all limiting 
	empirical measures $\al_f\in\mathcal{M}_1(E)$ the following is true: 
	
	For all sequences of conditionings
	$(\omega_{[2,N]})_{N\geq 2}$ with $\omega_{[2,N]}\in E^{N-1}$  whose empirical measures converge, 
	  $\frac{1}{N-1}\sum_{i=2}^N  \d_{\o_i}
	\rightarrow \al_f$,   the limit of the single-site conditional probabilities 
	\begin{align}\label{sequential}
	\lim_{N\rightarrow \infty}\mu_N(\omega_1\vert \omega_{[2,N]}) =:\g(\omega_1\vert \al_f). 
	\end{align}
	exists and does not depend on the choice of the sequence $(\omega_{[2,N]})_{N\geq 2}$.
	
	We say that $\al_f$ is a {\em bad empirical measure} of the model if \eqref{sequential} fails to hold, and different limits for $\mu_N(\omega_1\vert \omega_{[2,N]})$ can be constructed, 
	for two sequences of conditionings whose empirical measures converge to the same $\al_f$. 
\end{definition}

As a general consequence, if a mean-field model $\mu_N$ is sequentially Gibbs, the resulting 
{\em specification kernel} $\al_f\mapsto \g(\cdot \vert \al_f)$  is continuous as a self-map on the simplex 
$\mathcal{M}_1(\{-1,0,1\})$ (cf. \cite{zuijlen18},\cite{henning-kraaij-kuelske}). 
This makes clear that the sequential Gibbs property provides us with 
continuous dependence of conditional probabilities (here: in the limit), 
which is an essential requirement for 
Gibbsian theory on the lattice (\cite{enter-fernandez-sokal93},\cite{georgii-book}). 

Let us check our original model: 
The Curie-Weiss Widom-Rowlinson model with arbitrary a priori measure $\a$, at any repulsion $\b$, 
defined in terms of the sequence of finite-volume 
measures \eqref{WR} is indeed sequentially Gibbs, with specification kernel given by 
\begin{align*}
\g_{\b,\al}(\omega_1\vert \al_f)
= \frac{e^{- \beta \left( \mathds{1}(\omega_1=-1)\al_f(1)+\mathds{1}(\omega_1=1)\al_f(-1) \right)}\al(\omega_1) }{\sum_{\tilde\omega_1\in \{-1,0,1\}}e^{- \beta \left( \mathds{1}(\tilde\omega_1=-1)\al_f(1)+\mathds{1}(\tilde\omega_1=1)\al_f(-1) \right)}\al(\tilde\omega_1)}. 
\end{align*}
which is clearly a continuous function in $\al_f$ (in the usual Euclidean topology on the simplex). 
This formula follows from a simple rewriting of the Hamiltonian in exponent of \eqref{WR}  using $\mathds{1}(\omega_i\omega_j=-1)= \mathds{1}(\omega_i=-1)\mathds{1}(\omega_j=1)+ \mathds{1}(\omega_i=1)\mathds{1}(\omega_j=-1)$ and introducing  
the empirical measures on spins $2,\dots,N$.  

\subsection{Solution of the static Curie{-}Weiss Widom{-}Rowlinson model}

By standard large deviation arguments the pressure exists and equals 
\begin{equation}\label{equi: pressure1}
\begin{split}
		p(\beta,\alpha)&:= \lim_{N\rightarrow \infty} \frac{1}{N} \log Z_{N,\beta,\al} \cr
		&= \sup_{\nu\in \mathcal{M}_1(\{-1,0,1\})} (- \b \nu(1)\nu(-1)- I(\nu\vert \alpha ))
\end{split}
\end{equation}	
where $I$ denotes the relative entropy. Indeed, this follows from Varadhan's lemma and a rewriting 
of the Hamiltonian in the exponent of \eqref{WR} in terms of the empirical measure $\frac{1}{N}\sum_{i=1}^N \d_{\o_i}$ 
which is associated to a configuration $\o$. From Varadhan's lemma also follows that the negative of the quantity below the $\sup$, namely 
$\nu\mapsto  \b \nu(1)\nu(-1)+ I(\nu\vert \alpha )-C$ is the large deviation rate function for the distribution 
of the empirical measure under $\mu_{N,\b,\al}$, where the constant 
$C$ is determined such that the infimum becomes zero. 

Hence the maximizers in the sup in \eqref{equi: pressure1} (which will be non-unique at some $\b,\al$, namely when phase transitions  
of the model occur)
are the {\em typical empirical measures} at $\b,\al$. On these the distribution concentrates exponentially fast in $N$. 

It remains of course to discuss the behavior of the 
maximizers to get insight into its behavior, and in particular understand its transitions.  As a main piece of 
information we will obtain the following theorem. 

\begin{theorem}\label{thm: second order}
The symmetric model at any $\alpha(1)=\alpha(-1)>0$
has a second order phase transition driven by repulsion strength $\beta>0$ 
at the critical repulsion strength $\beta_c= 2+e \frac{\alpha(0)}{\alpha(1)}$.  
\end{theorem}

More detailed information can be obtained as follows.  
Let us parametrize the empirical spin distribution $\nu$ via two real coordinates $(x,m)\in [0,1]\times[-1,1]$, with the meaning of 
occupation density and  magnetization on occupied sites,  in the form
\begin{equation}\label{twofour}
\begin{split}
 \begin{pmatrix}
 \nu(-1) \\ \nu(0) \\ \nu(1) 
 \end{pmatrix}= 
 \begin{pmatrix}
 \frac{x}{2}(1-m) \\ 1-x \\ \frac{x}{2}(1+m)
 \end{pmatrix}
\end{split}.
\end{equation}	
Let us also parametrize the a priori measure  $\alpha$ via coordinates $(h,l)$, where 
$h:= \frac{1}{2} \log\left(\frac{\alpha(1)}{\alpha(-1)}\right)$ is a magnetic field-type variable describing the asymmetry 
of the model, 
and $l:=\log\frac{1-\alpha(0)}{\alpha(0)}$ describes a bias on occupation probabilities.  
The first step towards the closed solution of Theorem \ref{closedsolution} and which 
gives insight into the behavior of the model, 
is to rewrite the variational formula in \eqref{equi: pressure1} in the following representation  
in which a part for occupation density $x$, interacts with an Ising-type part for the magnetization $m$  
via an occupation-dependent coupling $\b x$.

\begin{lemma}\label{lem: pressure2} The pressure takes the form 
\begin{equation}\label{equi: pressure2}
\begin{split}
 &p(\beta,\alpha)=\log(\frac{1}{3}\alpha(0))+ \sup_{0\leq x,\vert m\vert\leq 1} \Bigl( \underbrace{-\frac{\beta x^2}{4}+x(l-\log(2\cosh(h)) -J(x)}_{\hbox{\small
 part for occupation density}}
 \\
& + x (\underbrace{\frac{\beta x m^2 }{4} +h m-I(m) }_{\hbox{\small Ising part at occupation-dependent coupling}}) \Bigr)
 \end{split}
 \end{equation}
  with entropies for spins and occupations given by 
 \begin{equation}\begin{split}\label{equi: entropies}
 &I(m)= \frac{1-m}{2}\log(1-m)+\frac{1+m}{2}\log(1+m) \\
 &J(x) = (1-x)\log(1-x)+ x\log{x}-x\log 2 
 \end{split}
 \end{equation}
\end{lemma}
To describe the relation between the $3$-dimensional parameter set given by $\b,\al$
and the typical values of  $\nu$ on the $2$-dimensional simplex obtained as maximizers,  
we treat $m$ as an independent parameter which allows us to obtain a closed solution as follows.

\begin{theorem}\label{closedsolution} Repulsion parameter $\b>0$, 
a priori measure $\alpha=\alpha(h,l)$,  
and possible typical values $\nu=\nu(m,x)$ of the empirical distribution, 
 are related via 
 \begin{align}\label{equi: asym beta}
\beta=\beta(m,\alpha)&= \frac{2}{m}(I'(m)-h)  (1+e^{-l+\log(\cosh(h)) +\frac{1}{m}(I'(m)-h)-mI'(m)+I(m)})\\
\label{equi: asym x} x=x(m,\alpha)&=(1+e^{-l+\log(\cosh(h)) +\frac{1}{m}(I'(m)-h)-mI'(m)+I(m)})^{-1} 
 \end{align}
 for $m\neq 0$. 
 \end{theorem}
Note that \eqref{equi: asym beta} describes all solutions to the stationarity 
	equation to carry out the maximization in \eqref{equi: pressure2}, and includes unstable 
	and metastable solutions, hence it describes the possible typical values of the 
	empirical distribution. 
	We can derive for instance critical exponents from this parametrization, see Theorem \ref{thm: critical beta} and \ref{thm: critical mag.}.

\subsection{Dynamical Gibbs-non Gibbs transitions, time-evolution of bad empirical measures}

%

Let us come to the time-evolution. 
We consider a stochastic time-evolution which exchanges $+$ and $-$ according 
to a temporal rate-$1$ Poisson process, and fixes the holes, 
independently at each site $i$.   
The corresponding single-site transition kernel which gives the probability 
to go from $a$ to $b$ in time $t$ at a site $i$ reads 
\begin{align}\label{trans}
p_t(a,b)= \frac{1}{2}(1+e^{-2t})\mathds{1}_{a=b\neq0}+\frac{1}{2}(1-e^{-2t})\mathds{1}_{a b =-1} 
+\mathds{1}_{a=b=0}.
\end{align}
for $a,b \in \{-1,0,1\}$ and $t>0$. 
We write  $\omega_{[1,N]}$ for a configuration at time $0$ and $\eta_{[1,N]}$ for a configuration
at time $t$.
The time-evolved measure on $N$ sites is defined by
\begin{align*}
\mu_{\beta, \al,t, N}(\eta _{[1,N]}) :=\sum_{\omega_{[1,N]} \in
	\Omega_N}\mu_{\beta, N}(\omega_{[1,N]}) \prod_{i=1}^N p_t(\omega_i,\eta_i)
\end{align*}

Then our main result on the dynamical Gibbs-non Gibbs transitions is as follows. 

\begin{theorem}\label{thm: time evolved} Consider the time-evolved Curie-Weiss Widom-Rowlinson model 
	at symmetric a priori measure $\al$, i.e. for which $\a(+)=\a(-)>0$, repulsion parameter $\b>0$
	and time $t>0$. 
	Then the following holds. 
	\begin{itemize}
		\item	For $\beta\leq2$ the time-evolved model is sequentially Gibbs for all $t>0$. 
		\item For $2<\beta \leq3$ the time-evolved model 
		is sequentially Gibbs iff $t<-\frac{1}{4}\log(1-\frac{2}{\beta})$.
		For $t\geq -\frac{1}{4}\log(1-\frac{2}{\beta})$ the set of bad empirical measures 
		is a line which grows with $t$. 
		\item For $\b >3$ there are three transition times $0<t_1(\b)<t_2(\b)<t_3=\frac{\log 3 }{4}$ such that the following holds:
		\begin{itemize}
			\item For $0\leq t<t_1(\beta)$ the model is sequentially Gibbs.
			\item At $t=t_1(\beta)$ 
			the model loses the sequential Gibbs property \\and 
			a pair of bad measures appears. 
			\item For $t_1(\b)<t<t_2(\b)$ the set of bad measures 
			consists of two disconnected curves.\\ (fig. \ref{fig: over t1}.)
			\item At $t=t_2(\b)$ the two curves touch. 
			\item 
			For $t_2(\b)<t<t_3$ 
			the set of bad empirical measures is Y-shaped
			(fig. \ref{fig: close t2},\ref{fig: close t0}).  
			\item 
			For $t\geq t_3$ 
			the set of bad empirical measures is a line which is growing with time.  
		\end{itemize}
	\end{itemize}
\end{theorem}
\begin{figure}
	\subfloat[$t<t_1(\beta)$\label{fig: below t1}]{\includegraphics[trim = 0mm 40mm 0mm 0mm  , clip,width=0.3\textwidth]{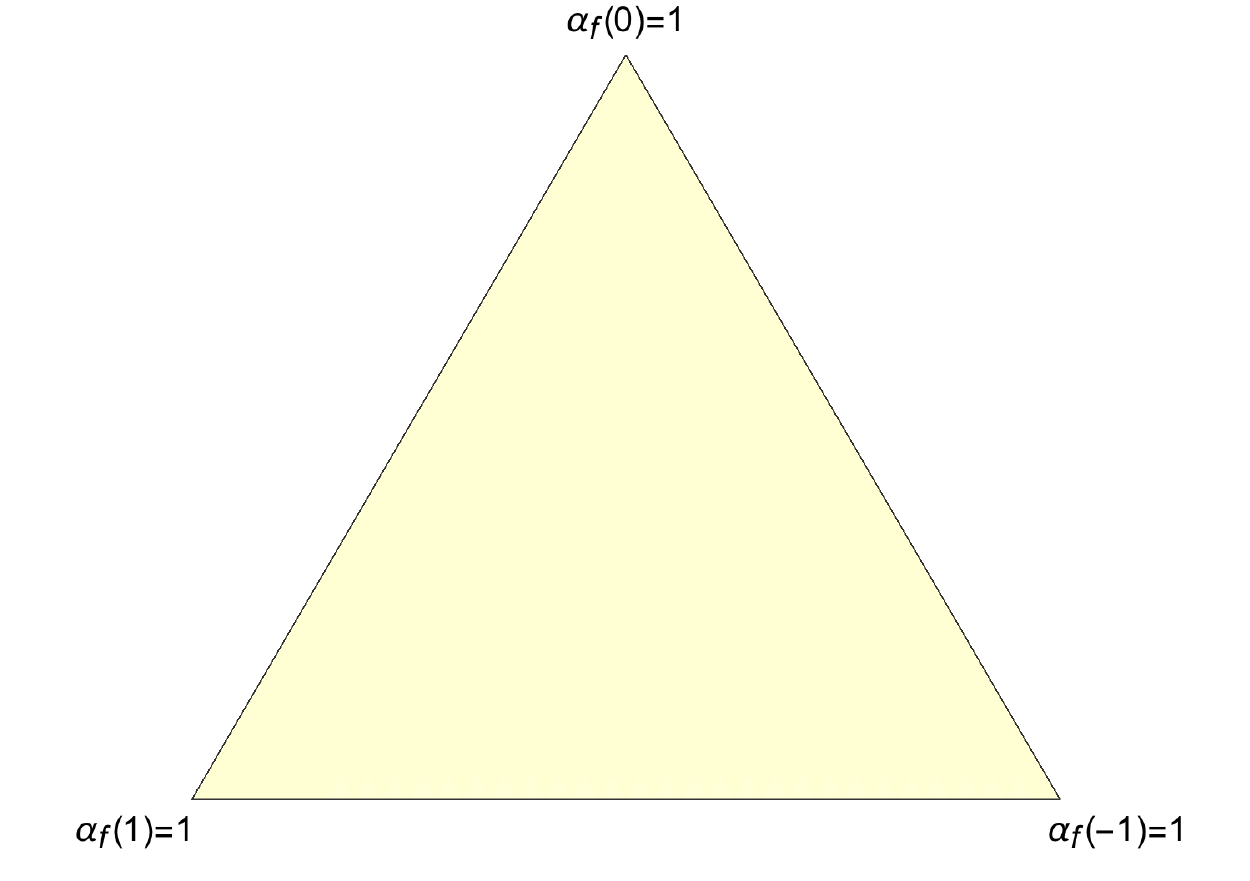}}
	\hfill %
	\subfloat[$t_1(\beta)\leq t< t_2(\beta)$\label{fig: over t1}]{\includegraphics[trim = 0mm 40mm 0mm 0mm  , clip,width=0.3\textwidth]{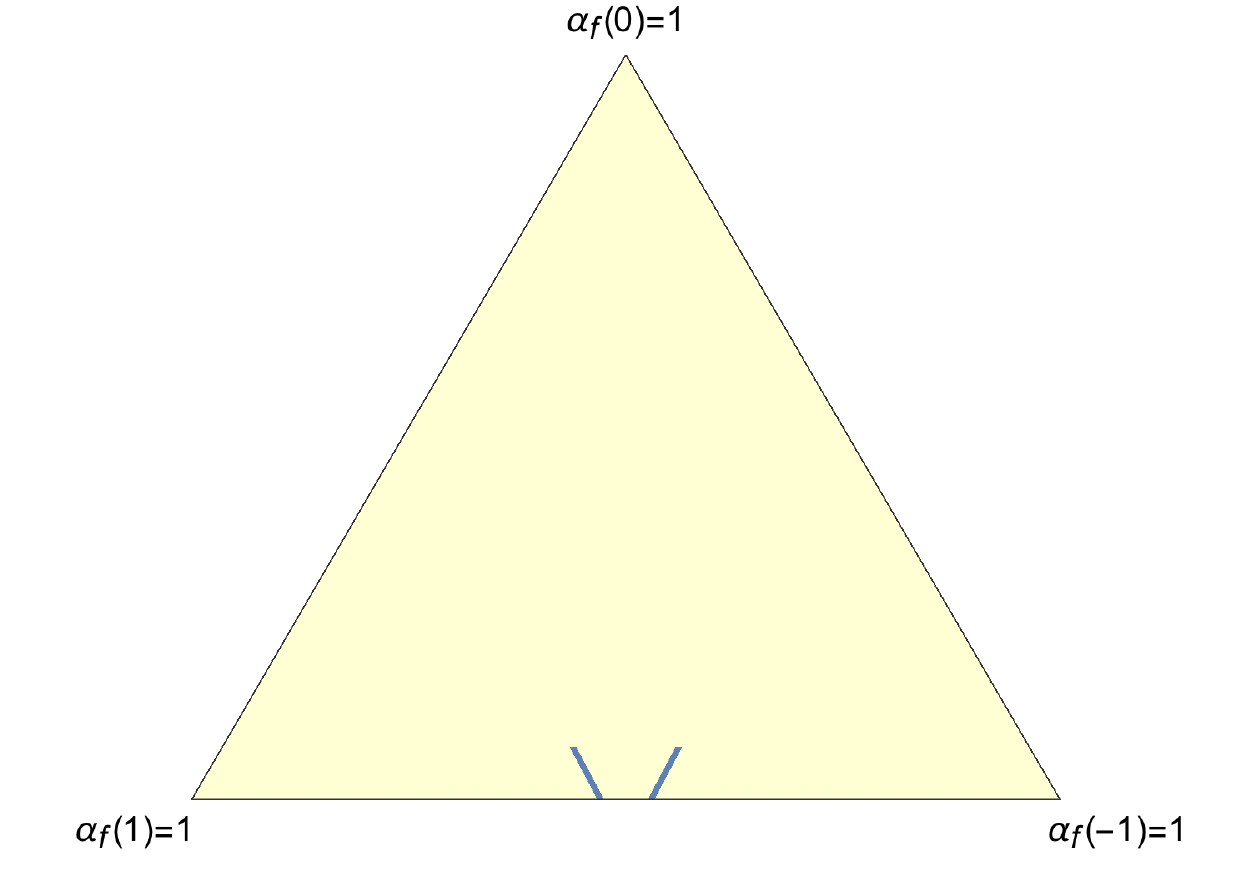}}
	\hfill 
	\subfloat[$t_2(\beta)\leq t<< t_3$\label{fig: close t2}]{\includegraphics[trim = 0mm 40mm 0mm 0mm  , clip,width=0.3\textwidth]{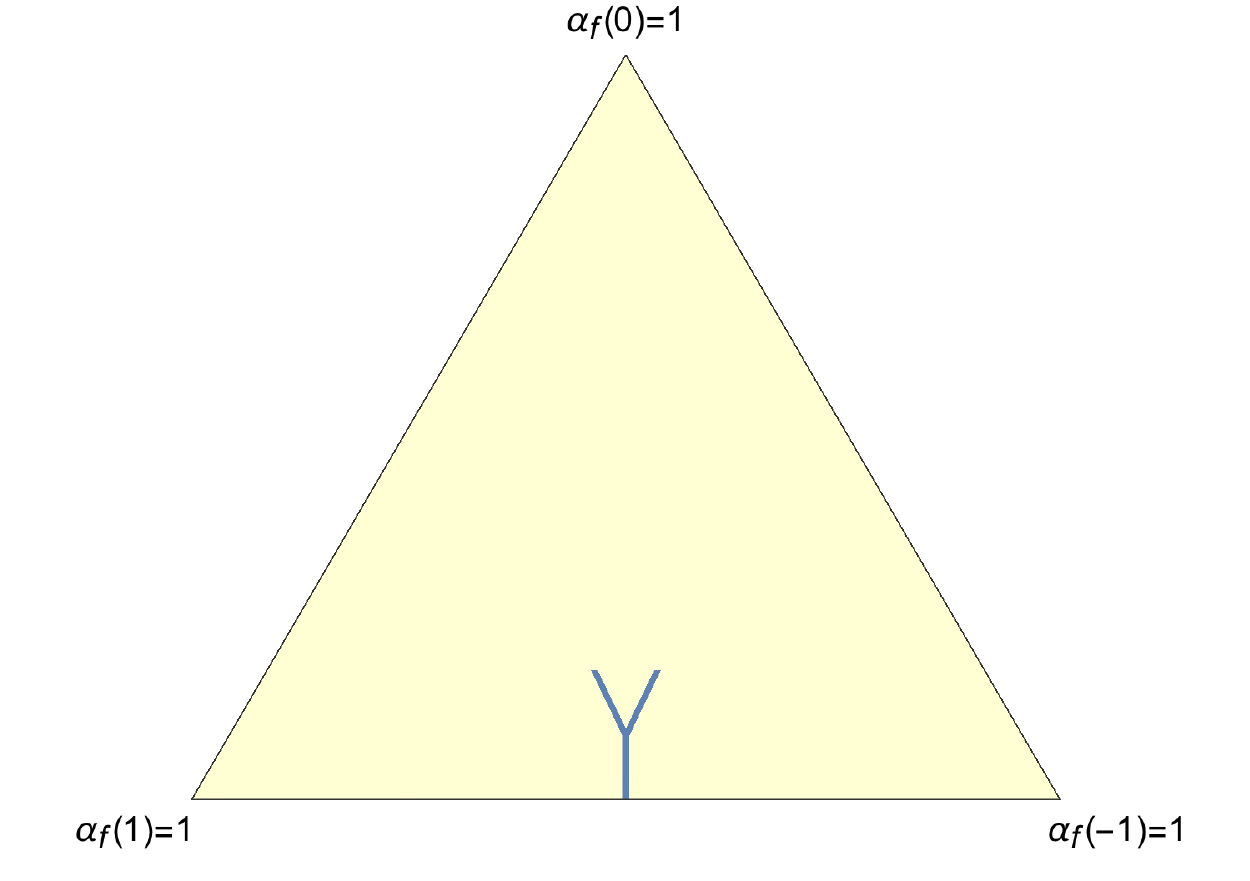}}
	\hfill
	\subfloat[$t_2(\beta)<< t< t_3$\label{fig: close t0}]{\includegraphics[trim = 0mm 40mm 0mm 0mm  , clip,width=0.3\textwidth]{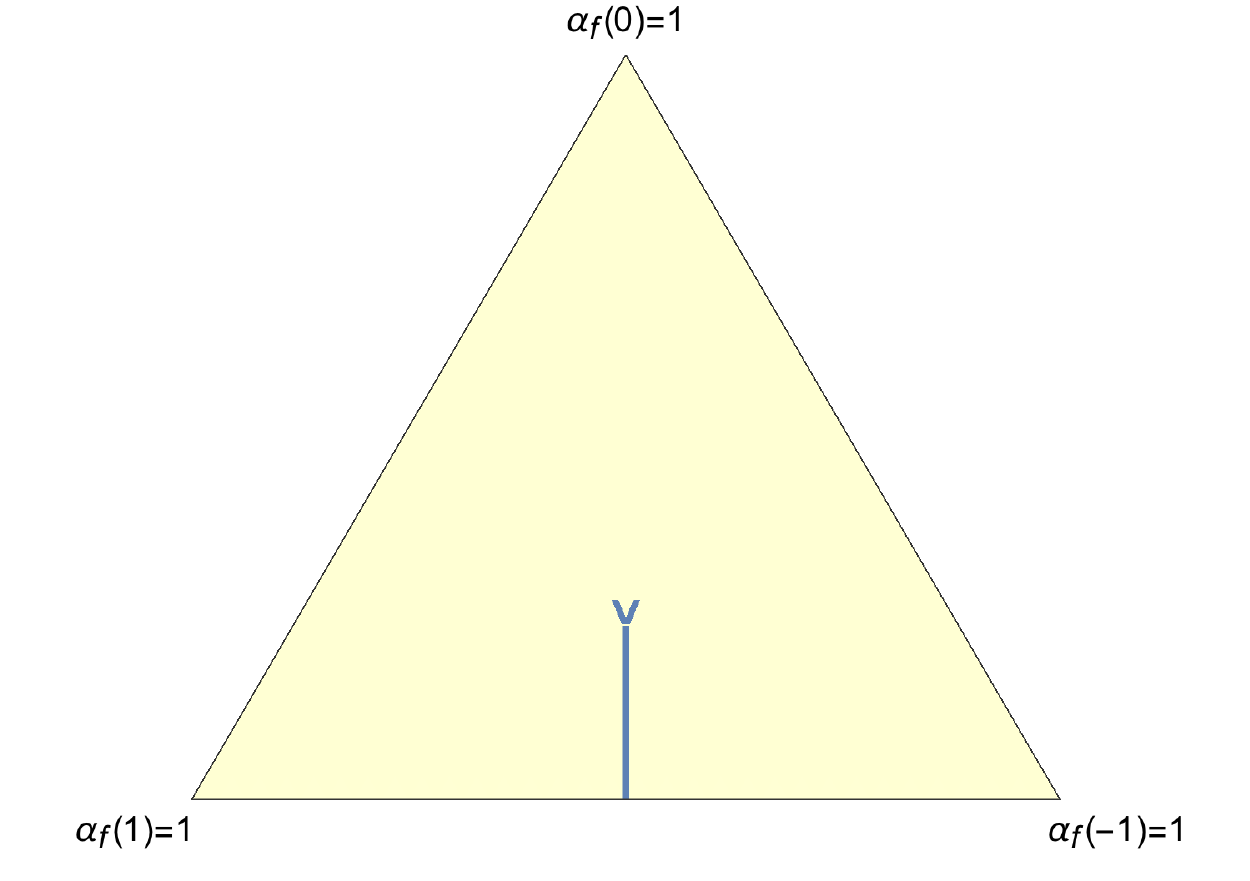}}
	\hfill 
	\subfloat[$t= t_{3}$\label{fig: equal t0}]{\includegraphics[trim = 0mm 40mm 0mm 0mm  , clip,width=0.3\textwidth]{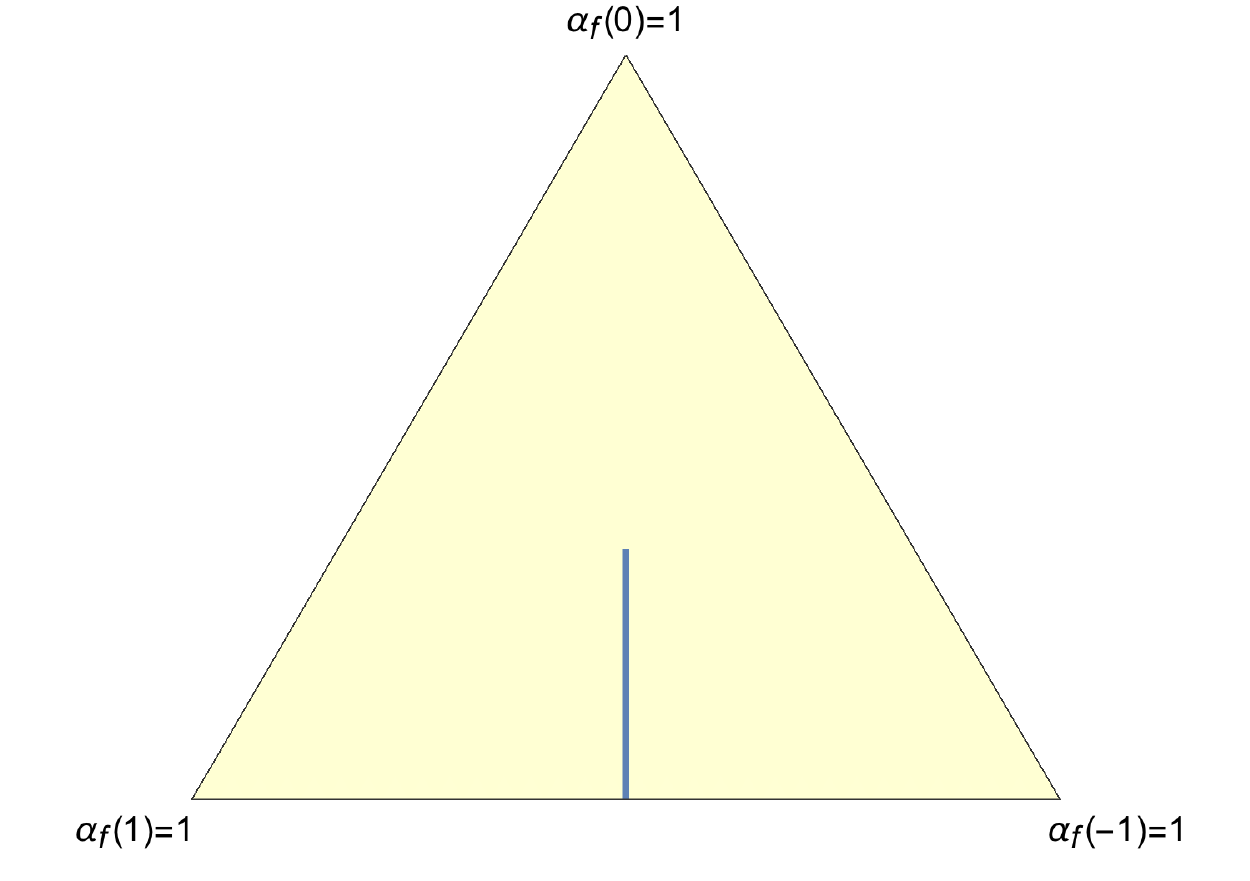}}
	\hfill 
	\subfloat[$t>>t_3$\label{fig: bigger t0}]{\includegraphics[trim = 0mm 40mm 0mm 0mm  , clip,width=0.3\textwidth]{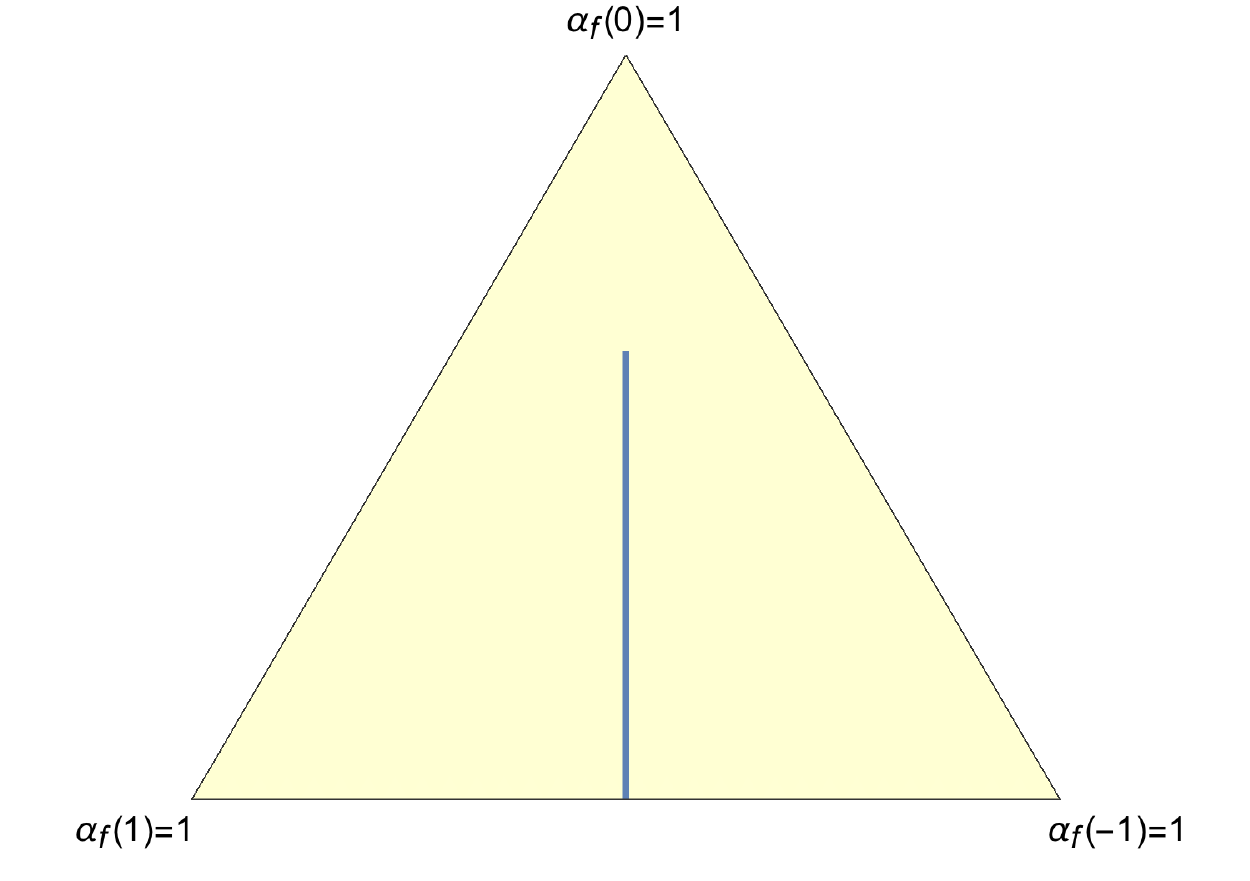}}
	\caption{Sets of bad empirical measures for different times and $\beta=5$}
\end{figure}
The above pictures describe the large $\b$-situation. For intermediate $2<\b\leq 3$, the bad empirical measures 
are described by a growing line, and qualitatively look like Figures \ref{fig: below t1}, \ref{fig: equal t0}, and \ref{fig: bigger t0}.
The transitions we just described do not depend on the a priori measure $\al$ 
as long as we assume that it is symmetric (which seems unusual but 
appears as a consequence of the nature of the 
dynamics which fixes the number of holes).

As the critical inverse temperature 
$\beta_c= 2+e \frac{\alpha(0)}{\alpha(1)}>2$ is always strictly bigger than the 
threshold $2$ for non-Gibbsian behavior, there is always non-Gibbsian behavior 
in the small-repulsion ("high-temperature") regime of the initial model. 

In the proof section we will present more information on the specification 
kernel of the time-evolved model $\g_{\b,\al,t}(\cdot | \al_f)$ in the parameter 
region of sequential Gibbsianness, see Lemma \ref{lem: limit exists}. 

We conclude our list of main results with a remark on typicality vs atypicality 
of bad empirical measures, or: Almost sure Gibbsianness.  
In analogy to the lattice situation we make the following definition.  
\begin{definition}\label{asq}
	We call a sequence of exchangeable measures $\mu_N\in \mathcal{M}_1(\O_N)$
	{\em almost surely sequentially Gibbs} iff there exists an $\e>0$ such that 
	$$\lim_{N\uparrow \infty}\mu_N\left( d\bigl(\frac{1}{N}\sum_{i=1}^N \d_{\o_i},  B\bigr)\geq \e\right)=1$$ 
	where $\o_{[1,N]}$ are distributed according to $\mu_N$,   
	$B$ is the set of bad empirical measures as in Definition \ref{def: seqG}, 
	and $d$ is the standard metric on  $\mathcal{M}_1(\{-1,0,1\})$. 
\end{definition}

In many examples,
the distribution of the empirical measures under
$\mu_N$ will even satisfy an LDP with rate $N$, and some rate function $\nu \mapsto K(\nu)$,  as $N$ tends to
infinity. In that case \ref{asq} is ensured by $\inf_{\nu, d(\nu,B)\leq \e}K(\nu)>0$.

With this definition we have in the case of our time-evolution the following proposition. 

\begin{proposition} The time-evolved model is almost-surely sequentially Gibbs, at all 
	parameters of the initial model $\b>0, \al$ and all times $t\in [0,\infty)$. 
\end{proposition}

This type of result  follows for non-degenerate (but possibly interacting) dynamics for Ising-systems 
by the principle of preservation of semi-concavity (see \cite[Theorem 2.11.]{kraaij-redig-zuijlen17}) 
In our present case where we have 
multivalued spins and degenerate dynamics \eqref{trans}
we include a proof for our specific model (see Section \ref{sec: atypi}). 
The situation is illustrated with the following plots. 
\begin{figure}[ht]\label{fig: atypic}\centering
\subfloat[$t=0.25,\b=2.8$\label{fig: atypic 2.8}]{\includegraphics[trim = 0mm 40mm 0mm 0mm  , clip,width=0.4\textwidth]{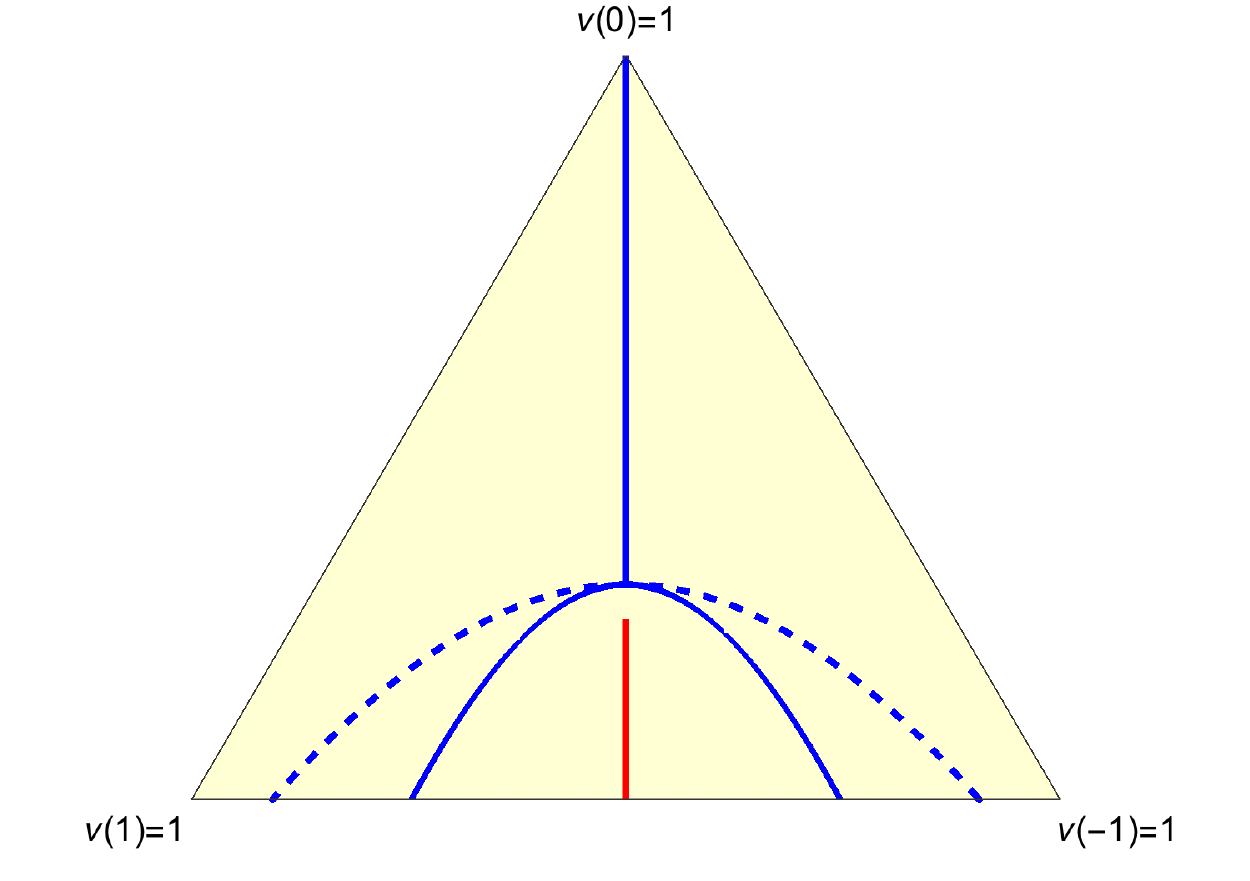}}
\subfloat[$t=0.111, \beta=4$\label{fig: atypic 5}]{\includegraphics[trim = 0mm 40mm 0mm 0mm  , clip,width=0.4\textwidth]{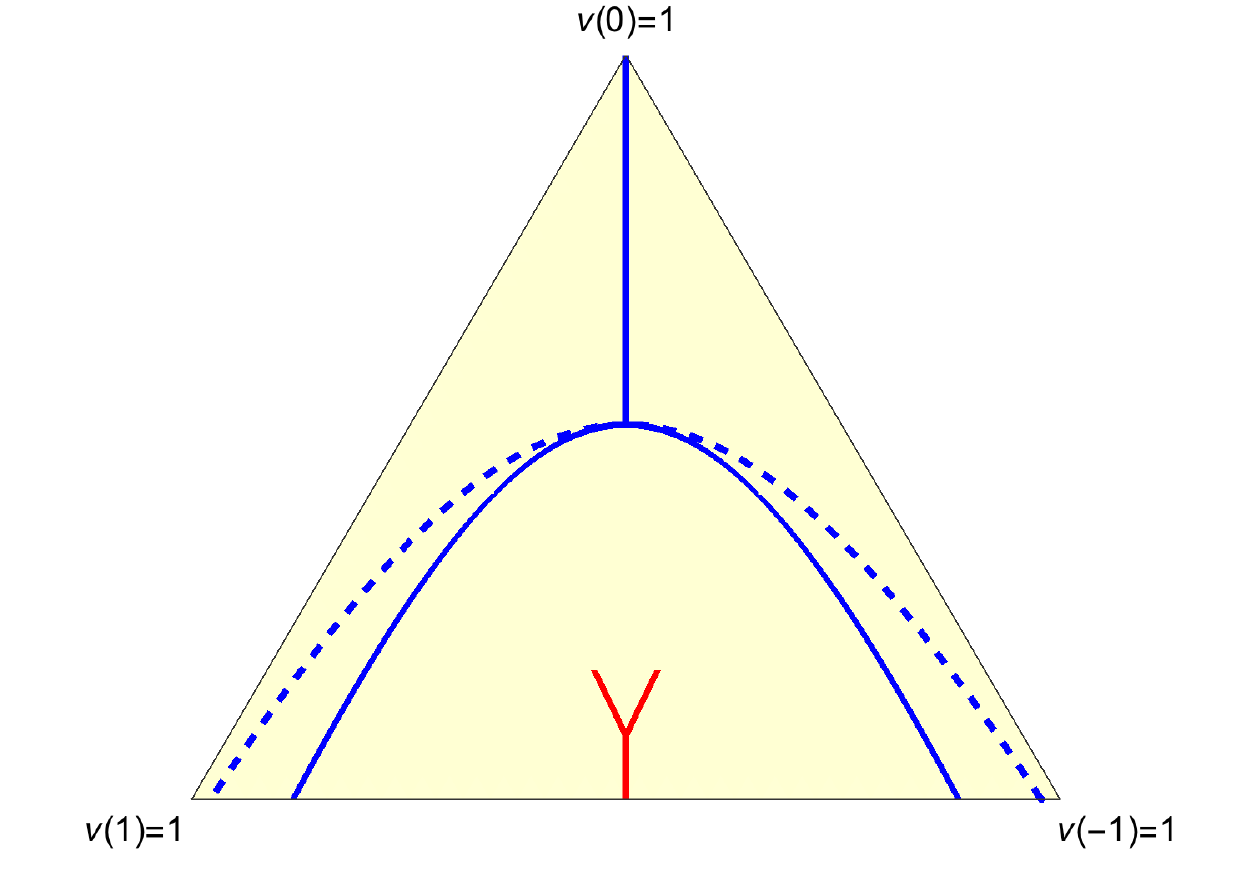}}\caption{
Bad empirical measures (red) and typical empirical measures at time $t$ (solid blue)}
\end{figure}
{The dashed blue line describes the locations of the asymmetric maximizers of \eqref{equi: pressure1} parametrized by $\al(0)$.  
Hence, all possible typical empirical measures  of the initial model for any possible hole density $\al(0)$ (including high values such that 
there is no broken symmetry),  lie above the dashed blue line. 
The solid blue line is the image of the dashed blue line after time-evolution (which contracts into the direction of the axis of symmetry). 
It therefore describes typical empirical measures of the time-evolved model.   
We will prove that the solid blue line will not intersect with the red set which is the set of bad empirical measures at time $t$.}

\section{The static model}

\subsection{Proofs for the main results}\label{sec: proofs static}

For the large deviation analysis we first consider only the symmetric model. This approach will not be enough to prove the whole Theorem \ref{thm: second order} but it will already give us the value of $\beta_c$. The first step is to prove \eqref{equi: pressure1}.
\begin{lemma}
	Let $\al\in\mathcal{M}_1(E)$ and $\beta>0$. Then the pressure $p$ of the Curie-Weiss WR model is equal to
	\begin{align}\label{equi: sup sym sc}
	p(\beta,\al)= \sup_{\nu\in \mathcal{M}_1(E)} (-H_\beta(\nu)- I(\nu\vert \al))
	\end{align}
	where $H_{\beta}(\nu)= \beta\nu(1)\nu(-1)$ and $I(\cdot\vert \al)$ is the relative entropy with respect to $\al$.
\end{lemma}
\begin{proof}
	The Hamiltonian of our model can be rewritten in terms of the empirical distribution 
	$L_N^{k} = \sum_{i=1}^N \mathds{1}(\omega_i=k)$ for $k\in E$, which 
	leads to a reformulation of the pressure
	\begin{align*}
	p(\beta,\al)&= \lim_{N\rightarrow \infty} \frac{1}{N} \log\left(\int_{\Omega_N} e^{-N  \beta L^1_{N}(\omega)L^{-1}_{N}(\omega)} \prod_{j=1}^N \al(d\omega_j) \right).
	\end{align*}
	Define a sequence $(\sigma_{[1,N]})_{N\geq 1}$ of i.i.d. random variables with law $\al$. Then the sequence $(\mathbb{P}_{L_N(\sigma_{[1,N]})})_{N\geq1}$ of laws for the empirical distribution satisfies a large deviation principle with speed $N$ and rate function $I(\cdot\vert \al)$ by Sanov's Theorem. Hence we have with Varadhan's Lemma that
	\begin{align*}
	p(\beta,\al)&= \lim_{N\rightarrow \infty} \frac{1}{N} \log\left(\int_{\mathcal{M}_1(E)} e^{-N\beta \nu(1)\nu(-1)} \mathbb{P}_{L_N(\sigma_{[1,N]})}(d\nu)\right)\\
	&=\sup_{\nu\in \mathcal{M}_1(E)} (-H_\beta(\nu)- I(\nu\vert \al)).
	\end{align*}
\end{proof}
Since the supremum is taken over a compact set it exists and will not lie on the boundary of $\mathcal{M}_1(E)$. This follows by the boundedness of $H_\beta$ and the properties of the relative entropy. To find all maximizers define $f(\nu):= -H_\beta(\nu)-I(\nu\vert \al)$ and take directional derivative in direction of the massless signed measure $\rho$ defined on $(E,\mathcal{E})$. This yields
\begin{align*}
\partial_{\rho}f(\nu)_{\vert_{t=0}} 
= -\beta(\nu(1)\rho(-1)+\nu(-1)\rho(1)) -\sum_{i\in\{-1,0,1\}} \rho(i) \log\left(\frac{\nu(i)}{\al(i)}\right).
\end{align*}
Now let $\nu_{m}$ denote maximizer of the function $f$. Taking $\rho(1) = -1,\rho(0)=1$ and $\rho(-1)=-1,\rho(0)=1$ gives the two equations
\begin{align}\label{equi: direc deri 1}
0&= \beta \nu_m(-1)+ \log\left(\frac{\nu_m(1)}{\al(1)}\right) - \log\left(\frac{\nu_m(0)}{\al(0)}\right)
\end{align}
and
\begin{align}\label{equi: direc deri 2}
0&= \beta \nu_m(1)+ \log\left(\frac{\nu_m(-1)}{\al(-1)}\right) - \log\left(\frac{\nu_m(0)}{\al(0)}\right).
\end{align} 
We have the following Lemma. \begin{lemma}
	Let $\al\in \mathcal{M}_1(E)$ be symmetric and $\beta>0$. Then there exists a 
	$\beta$-dependent solution $\nu_s\in\mathcal{M}_1(E)$ of the equations  \eqref{equi: direc deri 1} and \eqref{equi: direc deri 2} which is symmetric. Furthermore this is the only symmetric solution of the equations.
\end{lemma}
\begin{proof}
	First define $q:=\frac{\al(0)}{\al(1)}$ as shorthand notation. For symmetric $\nu_m$ \eqref{equi: direc deri 1} and \eqref{equi: direc deri 2} both are equivalent to 
	\begin{align}\label{equ: mean-field}
	e^{-\beta\nu_m(1)} &= \frac{\nu_m(1)}{1-2\nu_m(1)}q.
	\end{align}
	Since $e^{-x}$ is a decreasing function and $\frac{x}{1-2x}$ is an increasing function with a pole at $x=\frac{1}{2}$ there exists precisely one $x<\frac{1}{2}$ with $e^{-x} = \frac{x}{1-2x}$. This implies that there exists precisely one $\nu_s\in\mathcal{M}_1(E)$ depending on $\beta$ and $q$ which solves the above equations and is symmetric. Furthermore $\nu_s(1)$ is decreasing with increasing $\beta$ and $0<\nu_s(1)<\frac{1}{2+q}$. 
\end{proof}
Now we use independent coordinates  $\nu(-1), \nu(1)$
to parametrize the simplex.
In these coordinates the Hessian matrix of the function $f$ is given by

\begin{align*}
A_{f}(\nu)=(-1)\begin{pmatrix}
\frac{1}{\nu(1)} +\frac{1}{1-\nu(1)-\nu(-1)}& \beta + \frac{1}{1-\nu(1)-\nu(-1)}\\
\beta + \frac{1}{1-\nu(1)-\nu(-1)} & \frac{1}{\nu(-1)} +\frac{1}{1-\nu(1)-\nu(-1)}
\end{pmatrix}.
\end{align*}
We are seeking for a value of $\beta$ for which the type of the critical point at $\nu_s$ changes. 
The following lemma follows from a computation. 
\begin{lemma}\label{lem: matrix inde}
	The matrix $A_f(\nu_s)$ has an eigenvalue equal to zero if $\b$ equals $\beta_c= q e+2$.
The corresponding empirical measure $\nu_{s,c}$ is given by  $\nu_{s,c}(\pm1)= \frac{1}{qe+2}$.

\end{lemma}     
	If we set $q=0$ the critical repulsion strength is $2$. For this $q$ the hole density is zero and the Curie-Weiss WR model reduces
	 just to Curie-Weiss Ising model which has indeed critical inverse temperature $\beta_c=2$ \cite{ellis-newman78} (after taking 
	 into account our parameter choices). 
The next lemma is about the behavior of $A_f(\nu_s)$ for $\beta>\beta_c$.
\begin{lemma}\label{lem: mean sc indefinite}
	For all $\beta>\beta_c$ the matrix $A_f(\nu_s)$ has two eigenvalues different from zero, with different signs. Hence  
	$\nu_s$ is a saddle point.
\end{lemma} 
\begin{proof}
	It is easier to work with the diagonalised form of $A_f(\nu_s)$ which is equal to
	\begin{align*}
	D_f(\nu_s)=
	\begin{pmatrix}
	-(\frac{1}{\nu_s(1)}+\frac{2}{\nu_s(0)})-\beta& 0\\
	0& 	-\frac{1}{\nu_s(1)}+\beta&
	\end{pmatrix}.
	\end{align*}
	The first entry is always negative. Therefore we have to prove that
$\nu_s(1)>\frac{1}{\beta}$.  Indeed, assume $\nu_s(1)\leq \frac{1}{\beta}$.  Then by equation \eqref{equ: mean-field} we have
	\begin{align*}
	\frac{\nu_s(1)}{1-2\nu_s(1)}q =e^{-\beta\nu_s(1)} \geq e^{-1}
	\Leftrightarrow\nu_s(1) \geq \frac{1}{qe+2} = \frac{1}{\beta_c} > \frac{1}{\beta}
	\end{align*}
	which is a contradiction. 
\end{proof}
This all does not answer all relevant questions yet, but we have now an idea where the phase transition can occur. 
To complete the analysis we use a different approach where we will split the Hamiltonian of the 
model into a Curie-Weiss part on the occupied sites, with external magnetic field $h= \frac{1}{2} \log\left(\frac{\alpha(1)}{\alpha(-1)}\right)$,  and a part  which depends on the empirical occupation density. 

\begin{lemma}\label{lem: part. mean-asym}
	Let $\al \in \mathcal{M}_1(E)$, $\beta>0$ and $N \in \mathbb{N}$. Then it follows that
	\begin{align*}
	Z_{N,\beta,\al} 
	=&\sum_{\omega_{[1,N]}\in\Omega_N}e^{N L^0_N(\omega_{[1,N]})\log(\al(0)) 
		+\frac{1}{2}N(1-L^0_N(\omega_{[1,N]}))\log(\al(1)\al(-1))
	-\frac{\beta N}{4}(1-L^0_N(\omega_{[1,N]}))^2}\\
	&\times	\exp\Bigl(\frac{\beta}{4 N}\sum_{i,j \in S(\omega)}
	\omega_i\omega_j+ h\sum_{i\in S(\omega)} \omega_i\Bigl).
	\end{align*}
	where $S(\o)=\{i\,:\, \vert\omega_i\vert=1\}$ is the set of occupied sites.
\end{lemma}
\begin{proof} By a computation, using $1_{\omega_i\omega_j=-1}=-1/2(\omega_i\omega_j-1)$ for $\omega_i\omega_j\neq 0$. 
%
%
%
\end{proof}
With this representation of the partition function we can prove Lemma \ref{lem: pressure2} where we need the function $J(x)$ as defined in \eqref{equi: entropies}. Note that this function achieves its minimum at $2/3$ which is the typical size of an occupied volume 
when zeros, pluses, and minuses are drawn with equal weight.

\begin{proof}[Proof of Lemma \ref{lem: pressure2}]

We write the Curie-Weiss part of the partition function in terms of the empirical distribution of $+$ and $-$.
	Then again with Varadhan's Lemma and Sanov's Theorem  we obtain after ordering of the suprema over 
	the coordinates of $\nu$ that the pressure is given by 
	\begin{align*}
	p(\beta,\al)
	&= \sup_{0\leq \nu(0)\leq 1} \Biggl(\nu(0)\log\al(0) 
	+\frac{1}{2}(1-\nu(0))\log(\al(1)\al(-1))- \frac{\beta}{4}(1-\nu(0))^2 - \log(3\nu(0))\nu(0)\\
	&+\sup_{{\nu(1):}\atop{0\leq \nu(1)\leq 1-\nu(0)}} \Bigl[\frac{\beta}{4}(2\nu(1)+\nu(0)-1)^2+h (2\nu(1)+\nu(0)-1)\\
	&-  \log(3\nu(1)) \nu(1) - \log(3(1-\nu(1)-\nu(0)))(1-\nu(1)-\nu(0))\Bigr]\Biggr).
	\end{align*}
	We want to rewrite the inner supremum such that we can recognize the pressure of a Curie-Weiss model 
	at an effective temperature. 
	To do so, we write for the square bracket above 
	\begin{align*}
	&(1-\nu(0))\left[\frac{\beta(1-\nu(0))}{4}\left(\frac{2\nu(1)}{1-\nu(0)}-1\right)^2+h \left(\frac{2\nu(1)}{1-\nu(0)}-1\right)\right.\\
	&-  \log\left(2\frac{\nu(1)}{1-\nu(0)}\right) \frac{\nu(1)}{1-\nu(0)} - \log\left(2\left(1-\frac{\nu(1)}{1-\nu(0)}\right)\right)\left(1-\frac{\nu(1)}{1-\nu(0)}\right)
	\left.-  \log\left(\frac{3}{2}(1-\nu(0))\right)\right].
	\end{align*}
	 Note that $\tilde{\nu}$ with $\tilde{\nu}(\pm1)= \frac{\nu(\pm1)}{1-\nu(0)}$ defines a probability measure in $\mathcal{M}_1(\{-1,1\})$. 
	Comparing with the representation of the pressure of a Curie-Weiss model which, expressed in terms of the empirical distribution, 
	is given  
	by $$p^{CW}(\beta,h)= \sup_{0\leq\tilde \nu(1)\leq 1} (\frac{\beta}{2} (2\tilde \nu(1)-1 )^2+h(2\tilde \nu(1)-1)-\tilde \nu(1)\log(2\tilde \nu(1))-(1-\tilde \nu(1))\log(2(1-\tilde \nu(1))),$$ 
and changing to the parametrization \eqref{twofour} for the measure $\nu$, 
\eqref{equi: pressure2} follows. 
	\end{proof}
Now we are able to prove the representation theorem.
\begin{proof}[Proof of Theorem \ref{closedsolution}]
	By taking partial derivatives of the function inside of the sup in directions $x$ and $m$ we get the equations 
	\begin{equation}\label{mubn11}\begin{split}
	&0= l -\log (2\cosh h)- \beta x/2-J'(x)
	+m^2 x\beta/2 +hm-I(m) \cr
	\end{split}
	\end{equation}
	and 
	\begin{equation}\label{mubn12}\begin{split}
	&0=m x\beta/2 + h-  I'(m). \cr
	\end{split}
	\end{equation}
	Note that $J'(x)= \log  x - \log (1-x)-\log 2$ which is an invertible function in $(0,1)$ and $I'(m)= \frac{1}{2}\log(\frac{1+m}{1-m})$.
	We are interested in the behaviour of  $m, x$
	as a function of $\beta,h$ where we better
	treat $b=\beta x$ (instead of $\beta$) and  $h$ as independent 
	parameters. For $m\neq 0$ we have from the second equation
	\begin{equation}\label{mubn13}\begin{split}
	&b(m,\al)=\frac{2}{m}( I'(m)-h) \cr
	\end{split}
	\end{equation}
	which we recognise as an Curie-Weiss part of our model.
	We have from the first equation 
	\begin{equation}\label{mubn14}\begin{split}
	&x=x(b,\al,m)=(J')^{-1}\Bigl( l -\log (2\cosh h)- b/2
	+m^2 b/2 +h m-I(m) \Bigr)\cr
	\end{split}
	\end{equation}
	and $(J')^{-1}(x)= \frac{2}{2+e^{-x}}$.
	From the last two equations we get $x=x(m,\al)$, as in \eqref{equi: asym x}.
	From \eqref{mubn12} we get $\beta=\beta(x,m,\al)$. 
	Putting this together with \eqref{mubn14}
	we finally obtain \eqref{equi: asym beta} \end{proof}
From \eqref{mubn12} we see that $m=0$ can only be a candidate for a critical point if $h=0$. Therefore we suppose $h=0$ but then we are in the symmetric case and here we know that there exists always a unique symmetric solution and hence there exists an $x$ such that $m=0$ is critical point. \newline
Now fix $h$. Then we can get the desired curve $m$ vs. $\beta$ as  
a curve parametrized by $m$. 
\begin{figure}[!htb]\centering
	\subfloat[$h= 0,\al(0)=0.2$]{\includegraphics[width=0.3\textwidth]{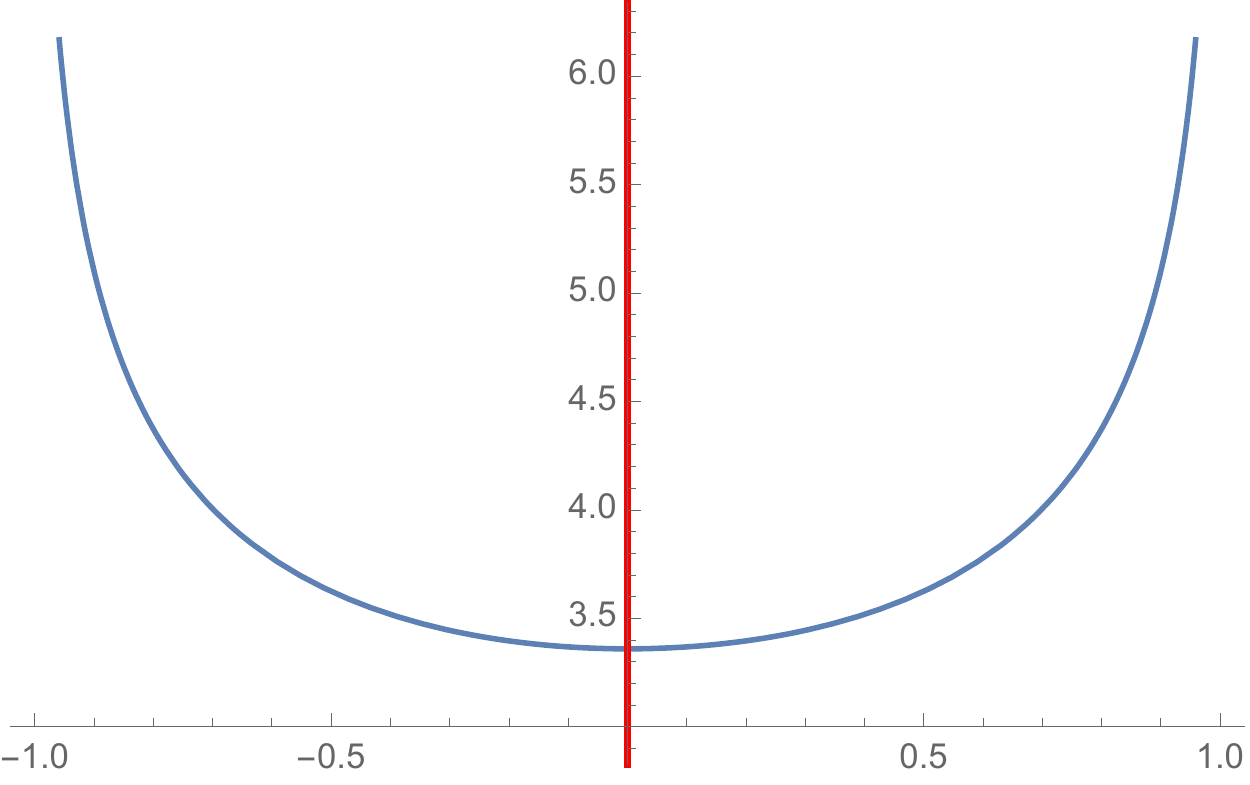}}\hspace{2cm}
	\subfloat[$h= 0.1438,\al(0)=0.3$]{\includegraphics[width=0.3\textwidth]{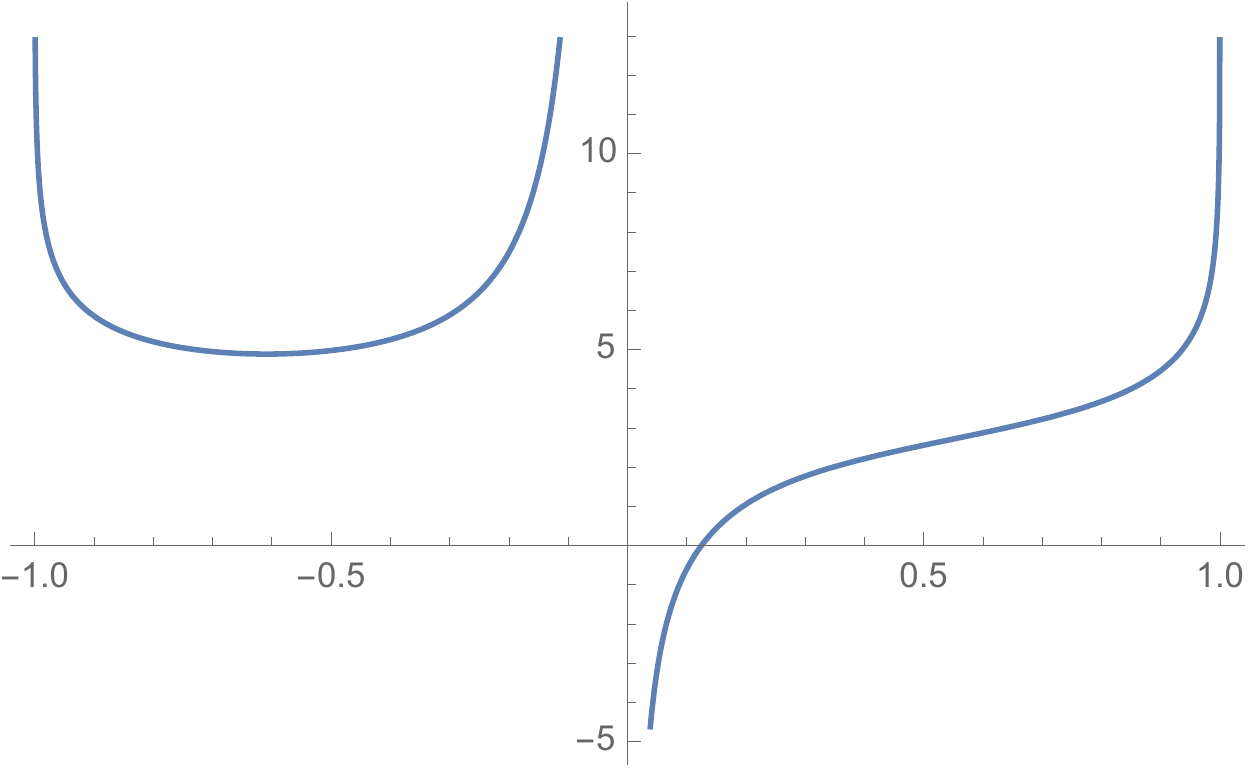}}
	\caption{Plots of $\beta$ vs $m$ at fixed $\al$}
\end{figure}
For $h=0$ the red line indicates that for every $\beta>0$ there is critical point with $m=0$. With the function $\beta(m,\al)$ we can describe the phase transition regimes of the symmetric Curie-Weiss WR model. 
\begin{lemma} \label{lem: mean sc sym }
	Let $\al\in\mathcal{M}(E)$ symmetric then for $\beta<\beta_c$ there exists no $m$ such that equation \eqref{equi: asym beta} holds and for $\beta>\beta_c$ there exist exactly two different values $m_1,m_2$ such that  $\eqref{equi: asym beta}$ holds. Furthermore $m_1$ and $m_2$ are related by $m_1=-m_2$.
\end{lemma}
\begin{proof}
	For a symmetric a priori measure $\al$ the function $\beta(m,\al)$ simplifies to
	\begin{align*}
	\beta(m,\al)= \frac{2}{m} I'(m)  (1+e^{-l+\frac{1}{m}I'(m)-mI'(m)+I(m)}).
	\end{align*}
	Since $\lim_{m\rightarrow 0}\frac{I'(m)}{m}=1$ we have 
	\begin{align}\label{equi: lim funk beta }
	\lim_{m\rightarrow 0}\beta(m,\al) = 2+q e = \beta_c
	\end{align}
	which is the critical beta for the symmetric model introduced earlier. Note that $e^{l}=2 q^{-1}$. 
	For $m>0$ the function $\beta(m,\al)$ is monotonically increasing. To see this write
	\begin{align*}
	\beta(m,\al)= \frac{2}{m} I'(m)  +\frac{2}{m} I'(m) (1-m)^{\frac{m-1}{2 m}} (m+1)^{\frac{m+1}{2 m}}q.
	\end{align*}
	The first summand is just an Curie-Weiss part $\beta^{Is}(m)= \frac{2}{m} I'(m)$ and it is known that this function is monotonically increasing on $m\in(0,1)$. For the remaining summand $\beta^{R}(m,\al):= \frac{2}{m} I'(m)  (1-m)^{\frac{m-1}{2 m}} (m+1)^{\frac{m+1}{2 m}} q$ we have to take a derivative which yields
	\begin{align*}
	\frac{\partial\beta^{R}(m,\al)}{\partial m}&= \frac{(1-m)^{-\frac{m+1}{ m}} (m+1)^{\frac{1-m}{2m}} q }{ 2m^3 }\left[4 m^2+\left(m^2-1\right) \log ^2(1-m)\right.\\
	&+\left.\left(m^2-1\right) \log ^2(m+1)-2 \left(m^2-1\right) \log (1-m) \log (m+1)\right].
	\end{align*}
	The desired monotonicity follows, if we can show that the last factor is bigger than $0$. This is equivalent to 
	\begin{align*}
	4m^2-(1-m^2)\log^2\left(\frac{1+m}{1-m}\right) >0
	\end{align*}
	which is again equivalent to
	\begin{align}\label{equi: posi m}
	2m+\sqrt{1-m^2} \log\left(\frac{1-m}{1+m}\right) >0
	\end{align} because $m>0$. The second derivative of the function $h(m) =\sqrt{1-m^2} \log\left(\frac{1-m}{1+m}\right)$ is equal to
	\begin{align*}
	h''(m)= -\frac{\log \left(\frac{1-m}{m+1}\right)}{\left(1-m^2\right)^{3/2}}
	\end{align*}
	which is strictly positive for all $m\in(0,1)$. Hence $h$ is strictly convex on $(0,1)$ and therefore $\tilde{h}(m):= 2m+\sqrt{1-m^2} \log(\frac{1-m}{1+m})$ is strictly convex on $(0,1)$. Since $\tilde{h}(0)=0$ and $\tilde{h}'(0)=4$ the convexity implies \eqref{equi: posi m}. This gives that $\beta(m,\al)$ is strictly monotonically increasing. \\
	Since $I'(-m)=-I'(m)$ and $I(-m)=I(m)$ it follows that $\beta(-m,\al)=\beta(m,\al)$. Hence $\beta(m,\al)$ is strictly monotonically decreasing on $(-1,0)$. This implies that for every $\beta<\beta_c$ no solution of \eqref{closedsolution} exists. For $\beta>\beta_c$ there exist exactly two solutions $m_1,m_2$ which are related by $m_1=-m_2$ because of $\beta(-m,\al)=\beta(m,\al)$. 
	
\end{proof}
Now we can prove the phase transition for the symmetric model.
\begin{proof}[Proof of Theorem \ref{thm: second order}]
	From Lemma \ref{lem: mean sc sym } we have that for all $\beta \leq\beta_c$ the only critical magnetisation is at $m=0$. This corresponds to the symmetric solution $\nu_s$ of \eqref{equ: mean-field}. We have proven that this solution is unique. Since the supremum of \eqref{equi: sup sym sc} is taken over a compact set we have that this symmetric solution is the unique maximizer. \\
	For $\beta>\beta_c$ the symmetric solution is a saddle point by Lemma \ref{lem: mean sc indefinite}. Again by Lemma \ref{lem: mean sc sym } there exist two critical magnetisations $m_1$ and $m_2$ with $m_1=-m_2$. By arguments as in the proof above we have
	$x(m_1,\al)=x(m_2,\al)$. This implies that there exist two extrema of \eqref{equi: sup sym sc} and both of them are global maximizers which follows by compactness and symmetry.
For all $\beta$ strictly below the critical repulsion strength $\beta_c$ there is only one extremum and no other critical points. Above $\beta_c$ there are two maximizers. So the model has a second order phase transition.
\end{proof}

For the asymmetric model we investigate the domains $m\in (0,1)$ and $(-1,0)$ separately. We fix now $h>0$ but by symmetry the following lemmas will also hold for $h<0$ with appropriate adjustments. First we prove that $\beta(m,h)\geq c>0$ if $m$ is negative.
\begin{lemma}\label{lem: exist delta}
	Let $\al \in \mathcal{M}(E)$ such that $\al^* >0$. Then there exists a $\delta_\al>0$ such that $\beta(m,\al)>\delta_\al$ for all $m\in(-1,0)$. 
\end{lemma}
\begin{proof}
	Since $(1+e^{-l+\log(\cosh(h)) +\frac{1}{m}(I'(m)-h)-mI'(m)+I(m)})$ is always bigger than $1$ we have only to consider the function $m\mapsto\frac{2}{m}(I'(m)-h)$. By the definition of $I'(m)$ we have to prove that there exists a $\delta_\al$ such that
	\begin{align*}
	\frac{2}{m}\left(\frac{1}{2}\log\left(\frac{1+m}{1-m}\right)-h\right) >\delta_\al
	\end{align*}
	for all $m \in(-1,0)$.  It is enough to prove $(\frac{1}{2}\log(\frac{1+m}{1-m})-h)<0$ since $\lim_{m\downarrow -1}\beta(m,\al)=\lim_{m\uparrow 0}\beta(m,\al) = \infty$ and $m<0$. Because $m<0$ the logarithm  $\log(\frac{1+m}{1-m})$ is negative and therefore the above inequality holds for all $m\in(-1,0)$. \\
\end{proof}
Clearly the function $\beta(m,\al)$ has some global minimizer on $(-1,0)$ and therefore there exists a best $\delta_\al$. But to find this minimizer is analytically quite hard.  The next lemma is about the domain $(0,1)$.
\begin{lemma}\label{lem: mono beta asym}
	Let $\al \in \mathcal{M}(E)$ such that $\al^*>0$. Then for every $\beta>0$ there exists a $m \in (0,1)$  such that equation \eqref{equi: asym beta} holds. It is the unique solution if $\beta<\delta_\al$. Furthermore $\beta(m,\al) <0$ for all $m\in \left(0,\frac{e^{2h}-1}{e^{2h}+1} \right)$ and $\beta(m,\al)\geq 0$ for all $m\in \left[\frac{e^{2h}-1}{e^{2h}+1},1\right)$. 
\end{lemma}
\begin{proof}
	For the first part of the proof it is enough to show by continuity that $\lim_{m\downarrow 0}\beta(m,\al)=-\infty$, $\lim_{m\uparrow 1}\beta(m,h)=\infty$ and that the function $\beta(m,h)$ is monotonically increasing. Since the second factor  $(1+e^{-l+\log(\cosh(h)) +\frac{1}{m}(I'(m)-h)-mI'(m)+I(m)})$ is always bigger than $1$ it does not play any role for the limiting behaviour. Since $\lim_{m\downarrow 0} \frac{2}{m}(I'(m)-h)=-\infty$ it follows that $\lim_{m\downarrow 0}\beta(m,h)=-\infty$. For the behaviour at $1$ it follows that
	$\lim_{m\uparrow 1} \frac{2}{m}(I'(m)-h)=\infty$ and therefore $\lim_{m\uparrow 1}\beta(m,h)=\infty$. The monotonicity follows by a similar computation as in the proof of Lemma \ref{lem: mean sc sym }.\\
	For the second part we need the root of $\beta(m,h)$ which is simply $m_r= \frac{e^{2h}-1}{e^{2h}+1}>0$ and its only root. Now take some $0<m<m_r$ for example $m= m_r/2$ then  $(\frac{1}{2}\log(\frac{1+\frac{1}{2}m_r}{1-\frac{1}{2}m_r})-h) = \arctanh(\frac{1}{2}\tanh(h))-h$ which is always negative for $h>0$. Similarly $(\frac{1}{2}\log(\frac{1+2m_r}{1-2m_r})-h) = \arctanh(2\tanh(h))-h$ which is always positive for $h>0$. This finishes the proof.
\end{proof}

\subsection{Curves of critical point for fixed external magnetic field}

Note that we have found the curves of critical points on the simplex of probability measures over $E$
as a function of $\beta$, for fixed $h$ and $\al(0)$. These are obtained via the explicitly known  function
$m\mapsto (x(m,\al),m)$.  

\begin{figure}[!ht]\centering

	\subfloat[$\beta=0$]{\includegraphics[trim = 0mm 40mm 0mm 0mm  , clip, width=0.35\textwidth]{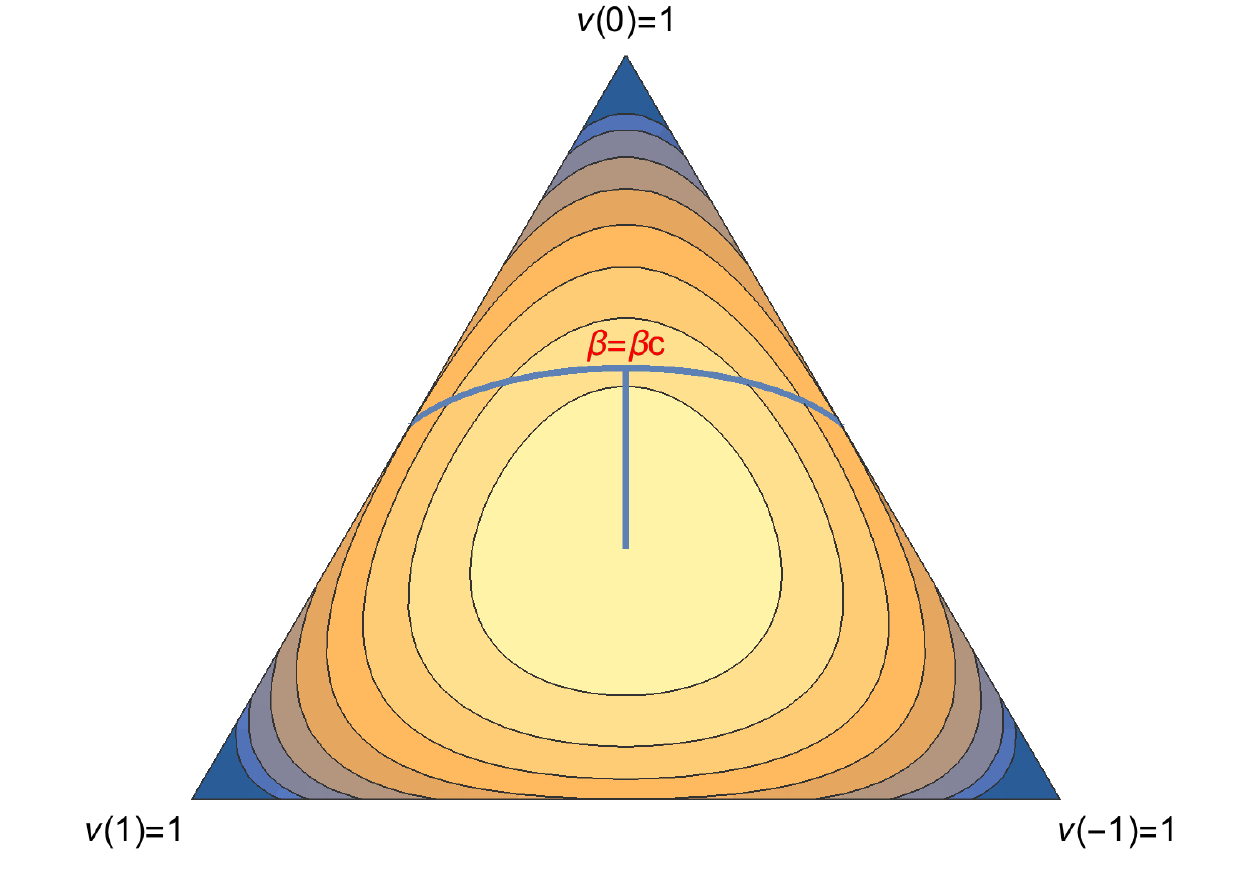}}
	\subfloat[$\beta=3.5<\beta_c$]{\includegraphics[trim = 0mm 40mm 0mm 0mm  , clip, width=0.35\textwidth]{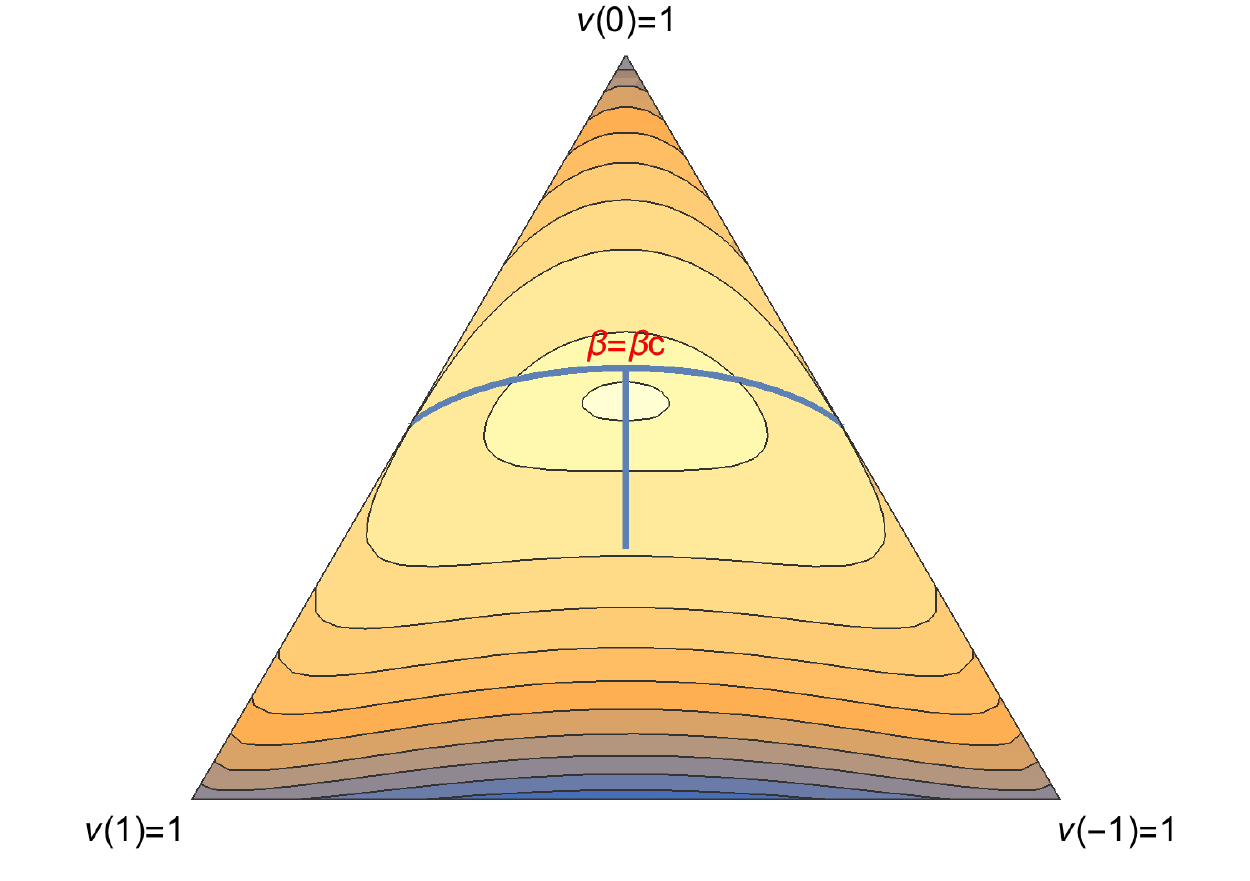}}\hfill 
	
\subfloat[$\beta=5>\beta_c$]{\includegraphics[trim = 0mm 40mm 0mm 0mm  , clip, width=0.35\textwidth]{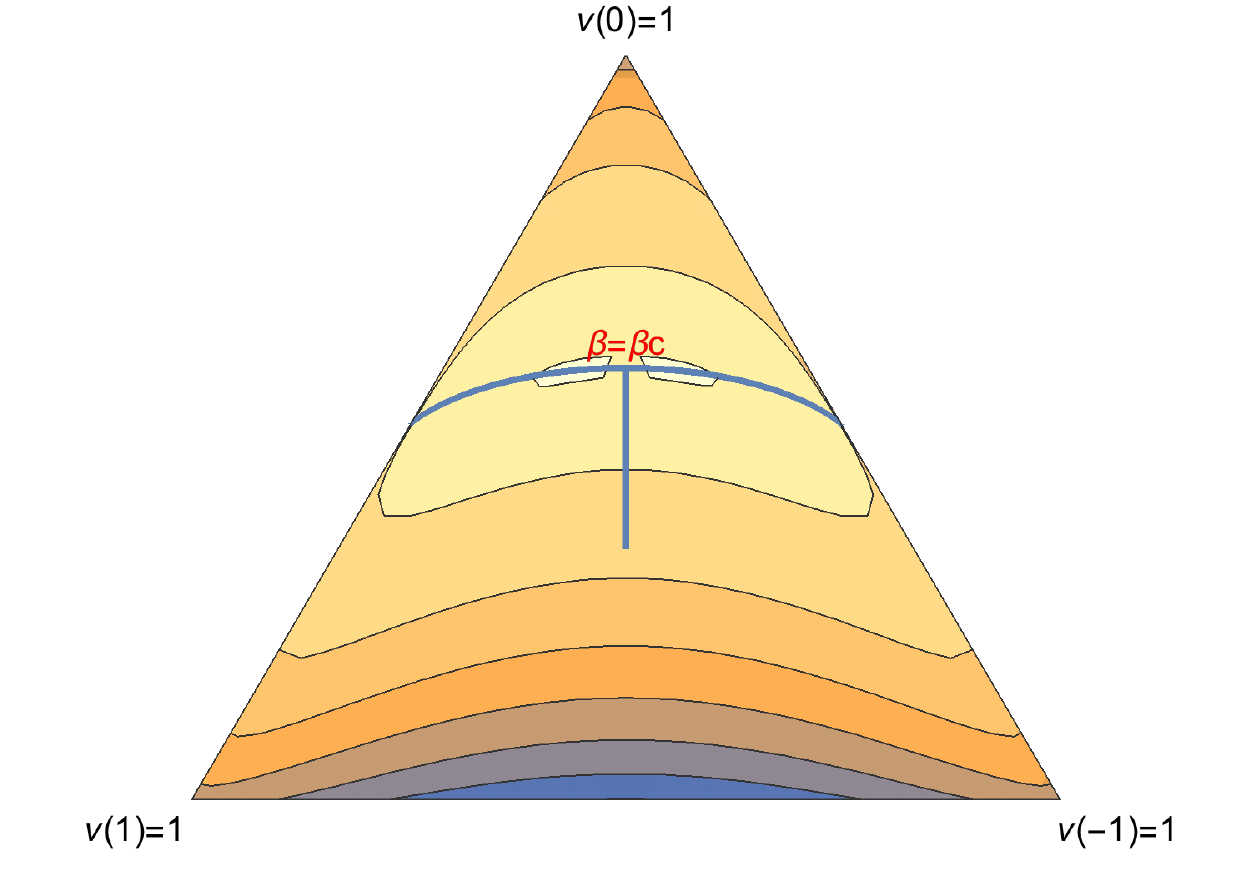}}
	\subfloat[$\beta>>\beta_c$]{\includegraphics[trim = 0mm 40mm 0mm 0mm  , clip, width=0.35\textwidth]{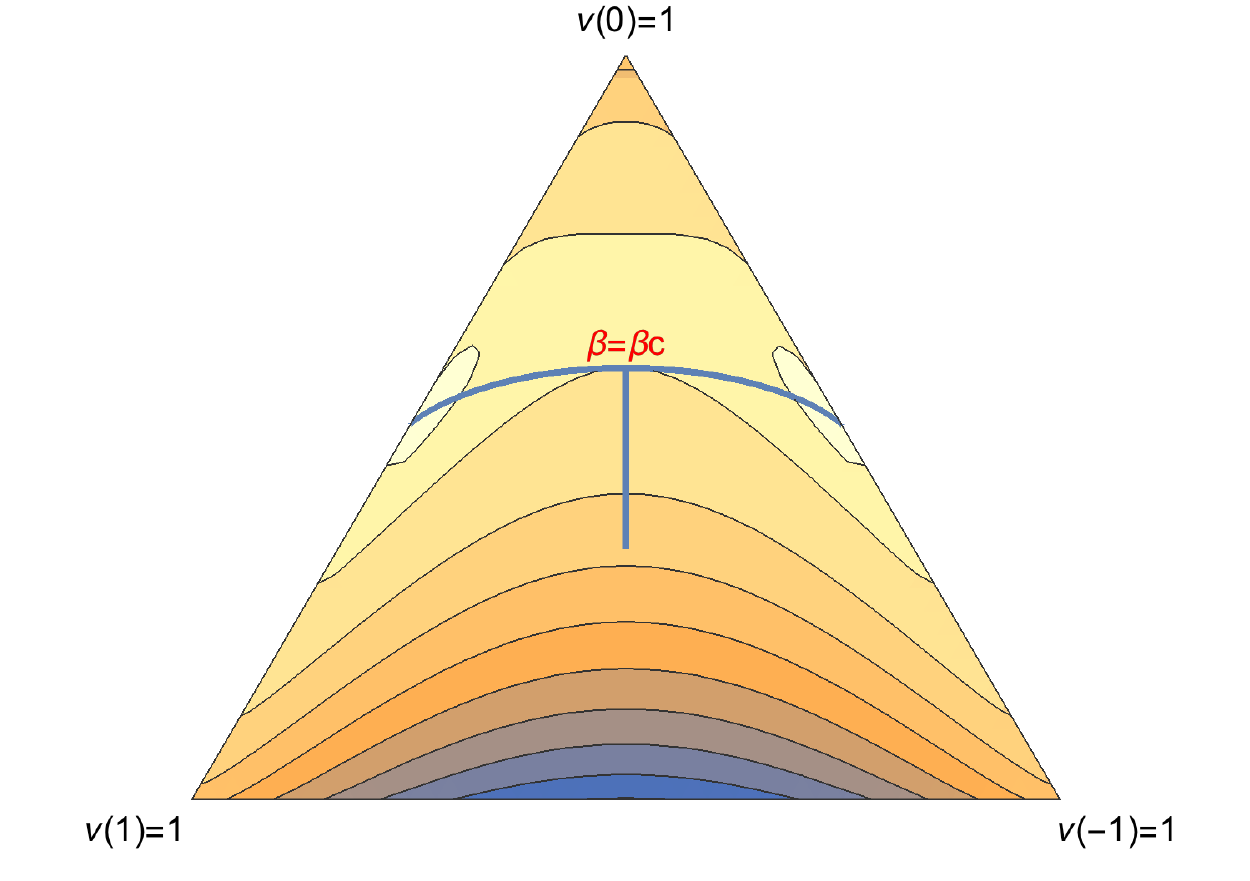}}\caption{Possible maximizers at fixed symmetric $\al$}
\end{figure}
The blue line gives the loci of the maxima of the function $-H_{\beta}(\nu)- I(\nu\vert \al)$ in dependence on $\beta$ for fixed  
symmetric $\al$. Here $\al$ is the equi-distribution which is also the maximizer in the first plot. If $\beta<\beta_c$ we see that only one maximizer exists which lies one the vertical part of the blue line. For $\beta>\beta_c$ the unique maximizer has split into two maximizers. \\
For asymmetric $\al$ the images look different. In the plots we have chosen $\al(1)=0.4$ and $\al(0)=\al(-1)=0.3$, which 
corresponds to an optimal value $\delta_h\approx 6.656$. The red line are the loci for possible other extrema.
\begin{figure}[!ht]\centering
	\subfloat[$\beta=0,h>0$]{\includegraphics[trim = 0mm 40mm 0mm 0mm  , clip, width=0.35\textwidth]{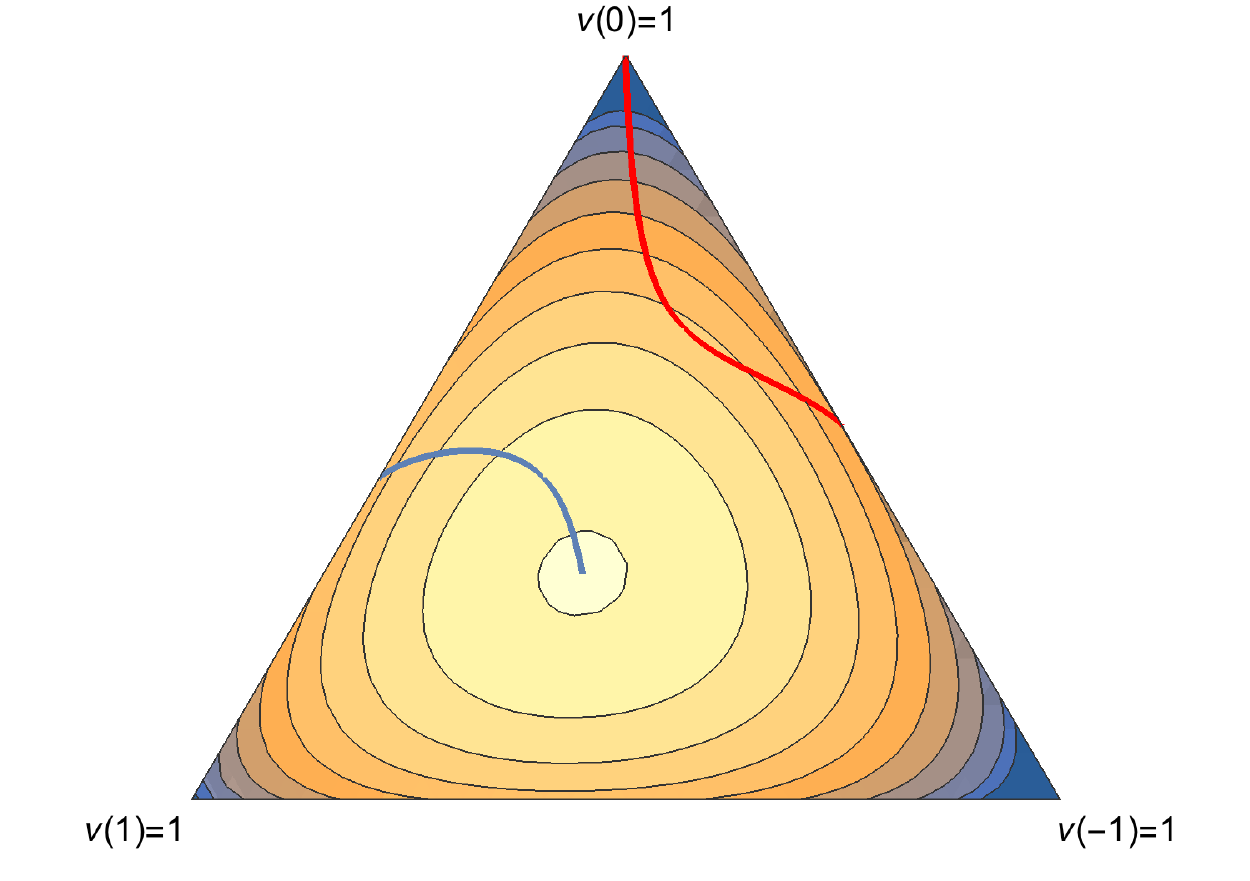}}
	\subfloat[$\beta= 6.7,h>0$]{\includegraphics[trim = 0mm 40mm 0mm 0mm  , clip, width=0.35\textwidth]{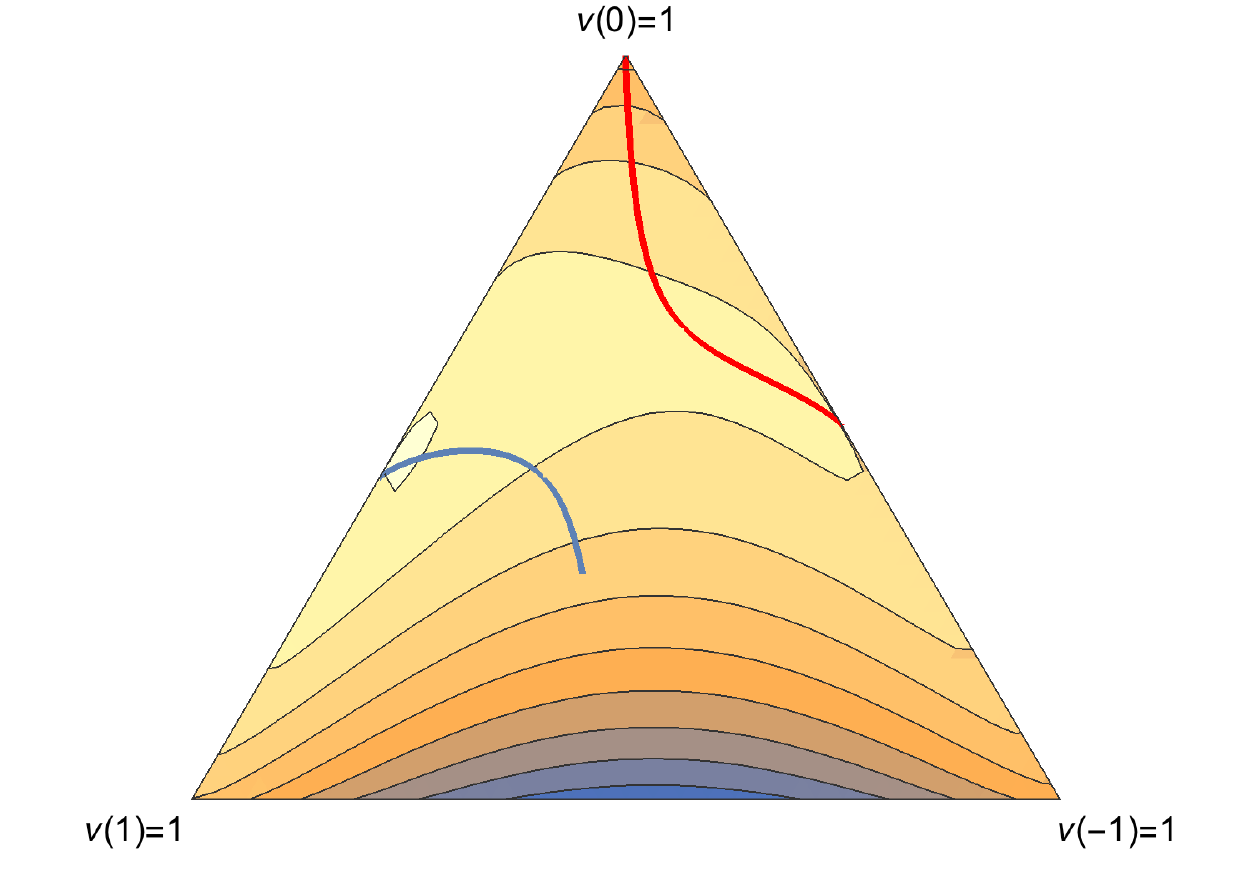}}\caption{Possible maximizers at fixed asymmetric $\al$}
\end{figure}

 \subsection{Critical exponents}
 
We saw before that the phase transition in the static model is of second order. As an additional piece 
of information, we investigate its behavior locally around the transition point, and recover (suitably defined) standard mean-field exponents. 
%
 \begin{theorem}\label{thm: critical beta}
 	Let $\al\in \mathcal{M}(E)$ symmetric. Then
 	\begin{equation}\label{equi: crit exp beta}
 	\lim_{\beta\downarrow \beta_c}\frac{m(\b)}{(\beta-\beta_c)^\frac{1}{2}}=
 	c 
 	\end{equation}	
 	for some constant $c\in(0,\infty)$. 
  \end{theorem}
 This means the critical repulsion exponent is equal $\frac{1}{2}$ which is the known 
 value of  the magnetization 
exponent of the Curie-Weiss model. Note that we have no explicit formula for $m(\beta)$ but if we restrict the function $\beta(m,0)$ on $m>0$ or $m<0$ it is bijective and $m(\beta)$ exists. Nevertheless we do not need an explicit formula for $m(\beta)$.
 \begin{proof}
 	The limit in \eqref{equi: crit exp beta} is equivalent to 
 	\begin{equation}
 	\lim_{m\downarrow 0}\frac{(\beta(m,\al)-\beta_c)}{m^2} = \frac{1}{c^2}
 	\end{equation}
 	since $\lim_{m\rightarrow 0} \beta(m,\al)=\beta_c$. Lets first take a look at the difference of the $\beta$'s where we will again recognise a Curie-Weiss part
 	\begin{align*}
 	\beta(m,\al)-\beta_c
 	=\left(\frac{2}{m} I'(m) -2\right) + \left(\frac{2}{m} I'(m) e^{-l +\frac{1}{m}I'(m)-mI'(m)+I(m)} -eq\right) .
 	\end{align*}
 	The first part is the same as in the Curie-Weiss model where we know that the critical exponent is $\frac{1}{2}$ and $\lim_{m\rightarrow 0} =\frac{\frac{2}{m} I'(m) -2}{m^2}=\frac{2}{3}$, cf. \cite{friedli_velenik_2017}. For the rest we can use again the function $\beta^R(m,\al)$ and by the same arguments as for \eqref{equi: lim funk beta } we have $\lim_{m\rightarrow 0} \beta^R(m,\al)= eq$. We will prove the rest of the statement with L'Hospital's rule where the the first two derivatives of $\beta^R$ are needed. The first can be found above
 	and the second is 
 	\begin{align*}
 	\frac{\partial^2\beta^{R}(m,\al)}{\partial m^2}=&\frac{(1-m)^{-\frac{3 m+1}{2 m}} (m+1)^{\frac{1-3m}{2m}} }{q^{-1}2m^5} \left[16 m^5-\left(m^2-1\right)^2 \log ^3(1-m)+\left(m^2-1\right)^2 \log ^3(m+1)\right.\\
 	&+4 \left(m^2-1\right)^2 m \log ^2(m+1)+\left(m^2-1\right)^2 \log ^2(1-m) (4 m+3 \log (m+1))\\
 	&-\left(m^2-1\right) \log (1-m) \left(12 m^2+3 \left(m^2-1\right) \log ^2(m+1)	+8 \left(m^2-1\right) m \log (m+1)\right)\\
 	&\left.+12 \left(m^2-1\right) m^2 \log (m+1)\right].
 	\end{align*}
 	We need that the first derivative converges to zero and the second to some constant bigger than $0$. Note that $\lim_{m\rightarrow 0} (1-m)^{-\frac{m+1}{2 m}} (m+1)^{\frac{1-m}{2m}} q  = eq$ and $\lim_{m\rightarrow 0}(1-m)^{-\frac{m+1}{2 m}} (m+1)^{\frac{1-m}{2m}} q =e^3q$. Therefore we need only consider the sums inside the brackets. Hence for the first derivative we have to investigate
 	\begin{align*}
 	&m^{-3}	\left(4 m^2+\left(m^2-1\right) \log ^2(1-m)+\left(m^2-1\right) \log ^2(m+1)-2 \left(m^2-1\right) \log (1-m) \log (m+1)\right)\\
 	&=\frac{1}{m^3}\left(4m^2-\log\left(\frac{1+m}{1-m}\right)^2\right) + O(m).
 	\end{align*}
 	Define $g(m):= 4m^2-\log^2(\frac{1+m}{1-m})$  then $g$ and the first 3 derivatives have limit $0$  which implies $\lim_{m\rightarrow 0} \frac{g(m)}{m^3} =0$. This gives $\lim_{m\rightarrow 0}\frac{\partial }{\partial m}\beta^R(m) =0$.\\
 	For the second derivative of $\beta^R$ only the last part is of interest which is asymptotically equal to  
 	\begin{align*}
 	\frac{1}{m^5}\left({\log^3\left(\frac{1+m}{1-m}\right)} +4m {\log^2\left(\frac{1+m}{1-m}\right)} -  12m^2 {\log\left(\frac{1+m}{1-m}\right)}\right)-8+O(m).
 	\end{align*}
 	Define the function $w(m):={\log^3(\frac{1+m}{1-m})} +4m {\log^2(\frac{1+m}{1-m})} -  12m^2 {\log(\frac{1+m}{1-m})}$ and this time we need the first 5 derivatives of this function. The first 4 derivatives have limit $0$ and the fifth
 	\begin{align*}
 	w^{(5)}(m)&= \frac{60 m^4+504 m^2+24 \left(25 m^2+16\right) m \log \left(\frac{m+1}{1-m}\right)+9 \left(5 m^4+10 m^2+1\right) \log ^2\left(\frac{m+1}{1-m}\right)+80}{-16^{-1}\left(m^2-1\right)^5}
 	\end{align*}
 	converges against $1280$ which yields $\lim_{m\rightarrow 0}\frac{\partial^2 }{\partial m^2}\beta^R(m) =\frac{4}{3}qe^3$. This implies $\frac{1}{c^2}= \frac{2}{3}(1+qe^3)$.
 \end{proof}
 The second critical exponent we are interested in describes the response 
 to tilting the a priori measure at the fixed critical repulsion strength. 
 It will be equal to $\frac{1}{3}$ like the magnetic field exponent in the Curie-Weiss model.
 \begin{theorem}\label{thm: critical mag.}
 	Let $\beta=\beta_c$ and $l\in\mathbb{R}$. Then 
 	\begin{equation}
 	\lim_{h\downarrow 0}\frac{m(\b_c,\al(h,l))}{h^\frac{1}{3}}= c
 	\end{equation}
 	for some constant $c\in (0,\infty)$.
 \end{theorem}
 Again we have no explicit formula for $m(\beta_c,h)$ but if we restrict the function $\beta(m,h)$ on $m>0$ it is again invertible and $m(\beta_c,h)$ exists for $h>0$.
 \begin{proof}
 The proof follows by the same idea as above.
 \end{proof}

\subsection{The antiferromagnetic model}
Here we assume that $\beta<0$. The model now attractive and the Hamiltonian favors asymmetric configurations. For the $h\neq0$ case Theorem \ref{closedsolution} is still true and we have for $h>0$ that $\beta(m,\al)$ is negative and monotonically increasing for all $m\in (0,\frac{e^{2h}-1}{e^{2h}+1})$ by Lemma \ref{lem: mono beta asym}. Hence the maximizer in \eqref{equi: sup sym sc} is unique for all $\beta<0$.\\
For  the symmetric model it follows by Theorem \ref{closedsolution} that only 
symmetric maximizers of \eqref{equi: sup sym sc} can exist. But since both functions in equation \eqref{equ: mean-field} are monotonically increasing we cannot say that there exists a unique symmetric solution. \\
Indeed, by using Lemma \ref{lem: pressure2} one get for the pressure of the symmetric model  
\begin{align*}
&p(\beta,\al)=\log\al(0)
+\sup_{0\leq x \leq 1} ( x ( l -\log (2))
- \frac{\beta}{4}x^2 -J(x))
\end{align*}
since the Curie-Weiss pressure is equal $0$ in this case. Define $V(x):=  x (l-\log(2))  - J(x)-  \frac{\beta x^2}{4}$ and by taking the first two derivatives  one get the bifurcation set 
\begin{align*}
B=\{(\b,l), \exists x \in (0,1): V'(x)=V''(x)=0\}.
\end{align*}
The both conditions give $\beta = -\frac{2}{x(1-x)}$ and $l=\log(2)+J'(x)-\frac{1}{1-x}$. Rewriting the latter equation in terms of $\al(0)$ and using the inverse of the repulsion strength one get a parametrisation of $B$ over $x$ by
\begin{align*}
B=\left\{\left(\frac{1}{\beta(x)}, \al_0(x)\right)\,:\, x\in(0,1)\right\}
\end{align*}
with $\frac{1}{\beta(x)}:= -\frac{x(1-x)}{2}$ and $\al_0(x):= \frac{1}{2\exp(J'(x)-\frac{1}{1-x})+1}$. In fig.\ref{fig: antiferro bifur} the blue line is the bifurcation set. Inside of the closed area the rate function has two maximizers and for fixed $\beta$ there exists a value of $\al_\beta(0)$ such that the two maximizers have equal height. This follows by the system of equation $V(x_1)-V(x_2)= 0$ and $V'(x_1)=V'(x_2)=0$ for $x_1\neq x_2$. The corresponding red curve in fig. \ref{fig: antiferro bifur} is called Maxwell-line which is given by the relation $\al_{\beta^{-1}}(0)= (e^{\frac{1}{4\beta^{-1}}}+1)^{-1}$ for $\beta^{-1}\in(-0.125,0)$ which is obtained by the following argument.

The function $V''$ is symmetric around $\frac{1}{2}$ and non-positive for all $x\in (0,1)$ iff $\beta\geq-8$. By the second property more than one maximizer may only exist if $\beta<-8$. Fix $\beta<-8$. By the symmetry of  $V''$ it has a primitive $f$ 
with $f(x)=-f(1-x)$. Clearly $V'$ is a primitive of $V''$ and depends only linearly on $l\in \R$. 
Hence there exists a $l_c$ such that $V'_{l_c}=f$. 
We will show that the value of $l_c$ defines the Maxwell-line where the two minima have equal 
depth.  
Note, by the choice of $l_c$,  $V_{l_c}$ is symmetric around $\frac{1}{2}$. 
 Since $V''$ has precisely two roots for $\beta<-8$ there exists a pair $x_1\neq x_2$ with $V_{l_c}(x_1)-V_{l_c}(x_2)= 0$ and $V'_{l_c}(x_1)=V'_{l_c}(x_2)=0$. To get the Maxwell-line we see that  the symmetry equation $V_{l_c}'(x)+V'_{l_c}(1-x) =0$ holds for all $x$ 
iff $ 0=  2l_c - \frac{\beta}{2}$,  and by the general definition of $l$ this equivalent to $\al_c(0)= (e^{\frac{\beta}{4}}+1)^{-1}$.
\begin{figure}[ht]
	\centering
	\includegraphics[width=0.33\textwidth]{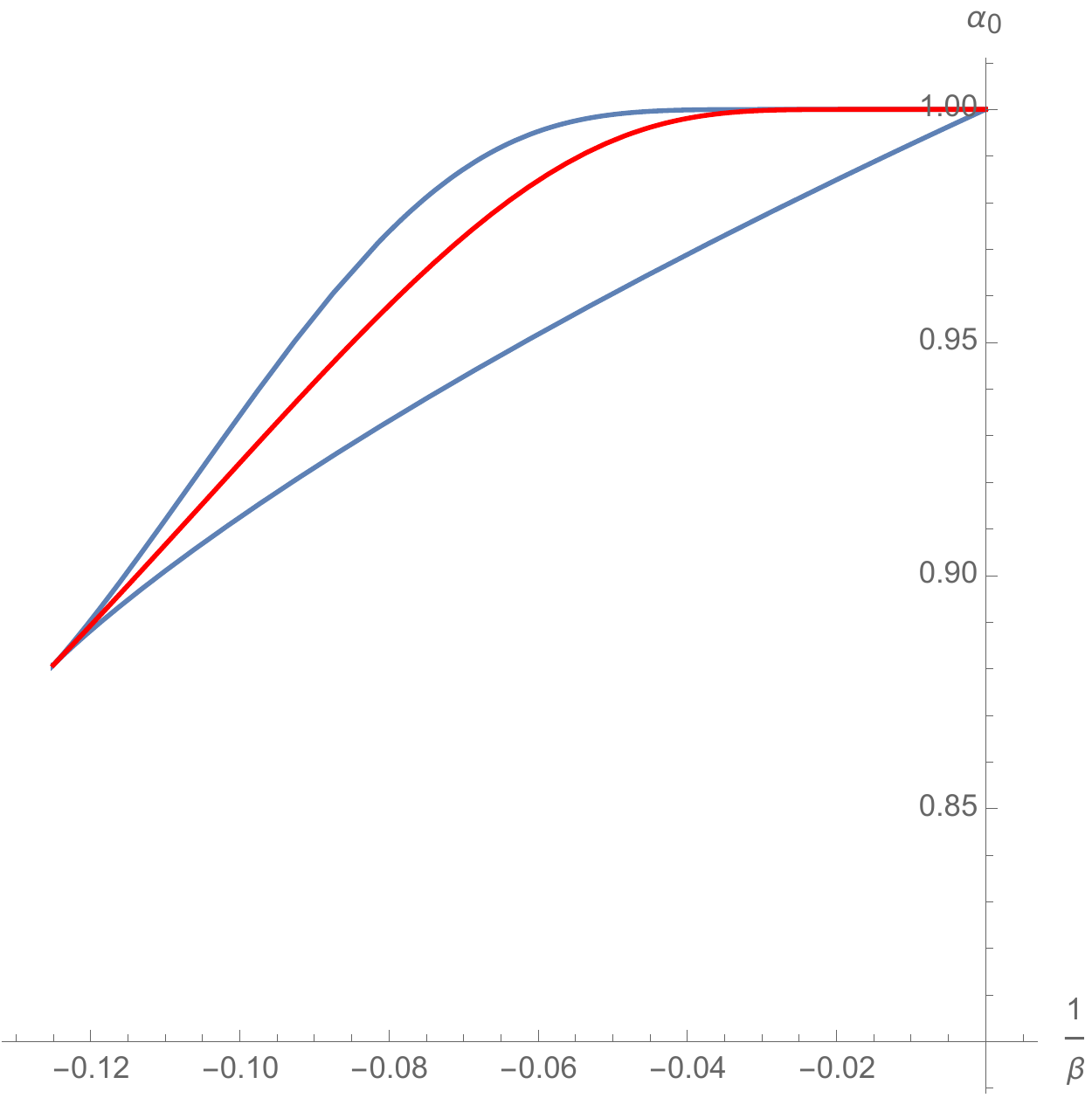}
	\caption{Bifurcation set}
	\label{fig: antiferro bifur}
\end{figure}

\section{Time evolution}
\subsection{Proofs for the main results}
In this part we give proofs for the dynamical model where will we use results from \cite{kuelske-le07}. There the authors investigated the Curie-Weiss model under stochastic time evolution via spin-flip. The core of their method was the usage of the so called constrained first-layer model which is a measure at time $0$ with a constraint coming from time $t$. We will use a similar approach and for this we need the next lemma. In the following the a priori measure $\al$ will always be symmetric and since its particular form has no effect on the results we will not mention it any more.
\begin{lemma}
	Let $\beta>0$ and $t>0$. Then the conditional probability of the time evolved measure can be written as
	\begin{align*}
	&\mu_{\beta, t, N}(\eta_1|\eta _{[2,N]}) \\
	&=\frac{\sum_{\omega_{[2,N]}}
		\phi_1(\eta_1, \frac{1}{N}\sum_{2\leq j \leq N} \delta_{\omega_j})
		\exp\Bigl(-\frac{\beta}{2 N}\sum_{2\leq i, j \leq N} 1_{\omega_i\omega_j=-1}\Bigr)\prod_{i=2}^N \al(\omega_i)p_t(\omega_i,\eta_i)}{\sum_{\omega_{[2,N]}} \phi_2( \frac{1}{N}\sum_{2\leq j \leq N} \delta_{\omega_j})
		\exp\Bigl(-\frac{\beta}{2 N}\sum_{2\leq i,j \leq N} 1_{\omega_i\omega_j=-1}\Bigr)\prod_{i=2}^N \al(\omega_i)p_t(\omega_i,\eta_i)}
	\end{align*}
	with 
	\begin{align*}
	\phi_1\left(\eta_1,\nu\right):&=\sum_{\omega_1\in E}
	e^{-\beta \left( \nu(1)\mathds{1}_{\omega_1=-1}+\nu(-1)\mathds{1}_{\omega_1=1}\right)}\al(\omega_1)p_t(\omega_1,\eta_1)\\
	\end{align*} 
and $\phi_2\left(\nu\right):=\sum_{\eta_1\in E}  \phi_1\left(\eta_1,\nu\right)$
	for positive measures $\nu\in\mathcal{M}^+(E)$. 
\end{lemma}
\begin{proof}
	Since the state space $E$ is finite we have 
	\begin{align*}
	\mu_{\beta, t, N}(\eta_1|\eta _{[2,N]}) = \frac{\mu_{\beta, t, N}(\eta_1\eta _{[2,N]})}{\mu_{\beta, t, N}(\eta _{[2,N]})}
	\end{align*}
	With the splitting 
	\begin{align*}
	\frac{\beta}{2 N}\sum_{1\leq i, j \leq N} 1_{\omega_i\omega_j=-1} = \frac{\beta}{2 N}\sum_{1\leq  i,j \leq N, \min\{i,j\} =1} \hspace{-0.5cm}1_{\omega_i\omega_j=-1}+\frac{\beta}{2 N}\sum_{2\leq i, j \leq N} 1_{\omega_i\omega_j=-1}
	\end{align*} 
	and definitions of $\phi_1$ and $\phi_2$ one can get the desired representation.
\end{proof}
Another way to write  $\phi_1$ for $\vert \eta_1 \vert =1$ which will be useful later is $\tilde\phi_1(\eta_1, \frac{1}{N-1} \sum_{2\leq j \leq N} \omega_j)$ where 
\begin{align*}
\tilde\phi_{1,n}(\eta_1, m)=\al(1)e^{-\beta\frac{N-1}{N} }\left(\eta_1e^{-2t}\sinh\Bigl(\beta\frac{N-1}{2N}
m\Bigr)+\cosh\Bigl(\beta\frac{N-1}{2N}
m\Bigr)\right)
\end{align*}
and for $\eta_1=0$ we define  $	\tilde\phi_1(0, m) = \al(0)$.\\
If we expand the fraction by 
\begin{align*}
\sum_{\omega_{[2,N]}}
\exp\Bigl(-\frac{\beta}{2 N}\sum_{2\leq i, j \leq N} 1_{\omega_i\omega_j=-1}\Bigr)\prod_{i=2}^N \al(\omega_i)p_t(\omega_i,\eta_i)
\end{align*}
one can see that the constrained first-layer model appears which will be defined now.
\begin{definition}
	Let $\beta>0$, $t>0$ and $\eta_{[1,N]} \in \Omega_N$. Then the constrained first-layer model with constraint $ \eta_{[1,N]} $ is defined by
	\begin{align*}
	\mu_{\beta, t, N}[\eta _{[1,N]}](\omega_{[1,N]})
	=		 \frac{	\exp\Bigl(-\frac{\beta}{2 N}\sum_{1\leq i, j \leq N} 1_{\omega_i\omega_j=-1}\Bigr)\prod_{i=1}^N \al(\omega_i)p_t(\omega_i,\eta_i)}{\sum_{\tilde\omega_{N}} \exp\Bigl(-\frac{\beta}{2 N}\sum_{1\leq i, j \leq N} 1_{\tilde\omega_i\tilde\omega_j=-1}\Bigr)\prod_{i=1}^N \al(\tilde\omega_i)p_t(\omega_i,\eta_i)}
	\end{align*}
	for $\omega_{[1,N]}\in \Omega_N$.
\end{definition}
This definition allows us to write for the conditional probability that
\begin{align*}
\mu_{\beta, t, N}(\eta_1|\eta _{[2,N]}) = \frac{\mu_{\beta, t, N}[\eta _{[2,N]}](\phi_1(\eta_1,\cdot))}{\mu_{\beta, t, N}[\eta _{[2,N]}](\phi_2(\eta_1,\cdot))}.
\end{align*}
The property of the transition kernel that no particle can be created or erased over time can also be expressed in terms of the set of occupied sites $S(\eta)=\{i \,:\, \vert \eta_i\vert=1 \}$. Define a new transition kernel  $$\tilde{p}_t(a,b):=\frac{1}{2}(1+e^{-2t})\mathds{1}_{a=b}+\frac{1}{2}(1-e^{-2t})\mathds{1}_{a\neq b }$$
but only for $a,b\in \{-1,1\}$. Then one can write for $\omega_N,\eta_N \in \Omega_N$ that
\begin{align}\label{equi: pt pt tilde}
\prod_{i=1}^N p_t(\omega_i,\eta_i) = \mathds{1}_{S(\omega_N)=S(\eta_N)} \prod_{i\in S(\eta_N)}\tilde{p}_t(\omega_i,\eta_i).
\end{align}
and for $a,b \in \{-1,1\}$
 \begin{align}\label{equi: transition kernel}
 \tilde{p}_t(a,b)=\frac{e^{ ab h_t}}{2 \cosh  h_t}, \; \; {\rm
 	with} \; \quad
 h_t=\frac{1}{2}\log\frac{1+e^{-2 t}}{1-e^{-2 t}}.
 \end{align}
With this relation we get the following lemma concerning the constrained first-layer model.
\begin{lemma}
	Let $\beta>0$, $t>0$ and $\eta_N\in \Omega_N$. Then we have
	\begin{align*}
	&\mu_{\beta, t, N}[\eta _{N}](\omega_{N}) \\
	&= \frac{1_{S(\omega_{N})=S(\eta_{N})}
		\exp\Bigl(\frac{\beta(\eta)}{4 |S(\eta)|}\sum_{ i,j \in S(\eta)} \omega_i\omega_j+ h_t \sum_{i\in S(\eta)}\omega_i \eta_i\Bigr)}{\sum_{\tilde{\omega}_N\in \Omega_N}1_{S(\tilde\omega_{N})=S(\eta_{N})}
		\exp\Bigl(\frac{\beta(\eta)}{4 |S(\eta)|}\sum_{ i,j \in S(\eta)} \tilde\omega_i\tilde\omega_j+ h_t \sum_{i\in S(\eta)}\tilde\omega_i \eta_i\Bigr)}\cr
	\end{align*}
	where $\beta(\eta):=\beta \frac{\vert S(\eta)\vert}{N}$. The restriction of $\mu_{\beta, t, N}[\eta _{N}]$ to $S(\eta)$ is a Curie-Weiss model on $\{-1,1\}^{S(\eta)}$.
\end{lemma}
\begin{proof}
	Take some bounded function $f:\Omega_N\rightarrow \R$. Then it follows by $\al(1)=\al(-1)$ and \eqref{equi: pt pt tilde} that
	\begin{align*}
	&\mu_{\beta, t, N}[\eta _{N}](f)\\
	&=		 \frac{	\sum_{\omega_{N}} f(\omega_N	)\exp\Bigl(-\frac{\beta}{2 N}\sum_{1\leq i, j \leq N} 1_{\omega_i\omega_j=-1}\Bigr)\al(0)^{ N-\vert S(\eta_N) \vert} \al(1)^{\vert S(\eta_n)\vert}\mathds{1}_{S(\omega_N)=S(\eta_N)} \prod_{i\in S(\eta_N)}\tilde{p}_t(\omega_i,\eta_i)}{\sum_{\omega_{N}} \exp\Bigl(-\frac{\beta}{2 N}\sum_{1\leq i, j \leq N} 1_{\omega_i\omega_j=-1}\Bigr) \al(0)^{ N-\vert S(\eta_N) \vert} \al(1)^{\vert S(\eta_n)\vert} \mathds{1}_{S(\omega_N)=S(\eta_N)} \prod_{i\in S(\eta_N)}\tilde{p}_t(\omega_i,\eta_i)}\\
	&=		 \frac{	\sum_{\omega_{S(\eta_N)}} f(\omega_{S(\eta_N)}0_{S(\eta_N)^c}	)\exp\Bigl(-\frac{\beta}{2 N}\sum_{ i, j \in S(\eta_N)} 1_{\omega_i\omega_j=-1}\Bigr) \prod_{i\in S(\eta_N)}\tilde{p}_t(\omega_i,\eta_i)}{\sum_{\omega_{S(\eta_N)}} \exp\Bigl(-\frac{\beta}{2 N}\sum_{ i, j \in S(\eta_N)} 1_{\omega_i\omega_j=-1}\Bigr)   \prod_{i\in S(\eta_N)}\tilde{p}_t(\omega_i,\eta_i)}.
	\end{align*}
	For $\tilde{p}_t$ we can use the characterisation  \eqref{equi: transition kernel}. The $\cosh$ term will cancel out and with $1_{\omega_i\omega_j=-1}=-1/2(\omega_i\omega_j-1)$ the measure can be written as
	\begin{align*}
	&\mu_{\beta, t, N}[\eta _{N}](f)= 		 \frac{	\sum_{\omega_{S(\eta_N)}} f(\omega_{S(\eta_N)}0_{S(\eta_N)^c}	)\exp\Bigl(\frac{\beta(\eta)}{4 |S(\eta)|}\sum_{ i,j \in S(\eta)} \omega_i\omega_j+ h_t \sum_{i\in S(\eta)}\omega_i \eta_i\Bigr)}{\sum_{\omega_{S(\eta_N)}}\exp\Bigl(\frac{\beta(\eta)}{4 |S(\eta)|}\sum_{ i,j \in S(\eta)} \omega_i\omega_j+ h_t \sum_{i\in S(\eta)}\omega_i \eta_i\Bigr)} 
	\end{align*}
	which is the desired representation.
\end{proof}
To find a nice representation of in terms of the constrained first-layer model let us  considers 
ratios of the conditional probabilities for different $\eta_1$
\begin{align*}
\frac{\mu_{\beta, t, N}(\bar\eta_1|\eta_{[2,N]})}{\mu_{\beta, t, N}(\eta'_1|\eta_{[2,N]})}
&= \frac{\sum_{\omega_{[2,N]}}
	\phi_1(\bar\eta_1, \frac{1}{N}\sum_{2\leq j \leq N} \delta_{\omega_j})
	\exp\Bigl(-\frac{\beta}{2 N}\sum_{2\leq i, j \leq N} 1_{\omega_i\omega_j=-1}\Bigr)\prod_{i=2}^N \al(\omega_i)p_t(\omega_i,\eta_i)}{\sum_{\omega_{[2,N]}}
	\phi_1(\eta'_1, \frac{1}{N}\sum_{2\leq j \leq N} \delta_{\omega_j})
	\exp\Bigl(-\frac{\beta}{2 N}\sum_{2\leq i, j \leq N} 1_{\omega_i\omega_j=-1}\Bigr)\prod_{i=2}^N \al(\omega_i)p_t(\omega_i,\eta_i)}\\
&= \frac{\mu_{\beta,t,N}[\eta_{[2,N]}](\phi_1(\bar\eta_1, \frac{1}{N}\sum_{2\leq j \leq N} \delta_{\omega_j}))}{\mu_{\beta,t,N}[\eta_{[2,N]}](\phi_1(\eta'_1, \frac{1}{N}\sum_{2\leq j \leq N} \delta_{\omega_j}))}.
\end{align*}
Indeed, by this we get the nice representation
\begin{align*}
\mu_{\beta, t, N}(\eta_1|\eta_{[2,N]}) = \frac{\mu_{\beta,t,N}[\eta_{[2,N]}](\phi_1(\eta_1, \frac{1}{N}\sum_{2\leq j \leq N} \delta_{\omega_j}))}{\sum_{\bar\eta_1\in\{-1,0,1\}}\mu_{\beta,t,N}[\eta_{[2,N]}](\phi_1(\bar\eta_1, \frac{1}{N}\sum_{2\leq j \leq N} \delta_{\omega_j}))}.
\end{align*}
and, since $\phi_1(0_1, \frac{1}{N}\sum_{2\leq j \leq N} \delta_{\omega_j}) =\al(0),$ we have
\begin{align}\label{equi: meas rep }
\mu_{\beta, t, N}(\eta_1|\eta_{[2,N]}) = \frac{\mu_{\beta,t,N}[\eta_{[2,N]}](\phi_1(\eta_1, \frac{1}{N}\sum_{2\leq j \leq N} \delta_{\omega_j}))}{\al(0)+	\sum_{\bar\eta_1\in\{-1,1\}}\mu_{\beta,t,N}[\eta_{[2,N]}](\phi_1(\bar\eta_1, \frac{1}{N}\sum_{2\leq j \leq N} \delta_{\omega_j}))}.
\end{align}
The convergence of the single-site conditional 
probabilities of $\mu_{\beta,t,N}$ appearing on the l.h.s. of the last equation, 
in the sense of Definition \ref{def: seqG}, is now completely determined by the convergence of $\mu_{\beta,t,N}[\eta_{[2,N]}]$, 
as  the empirical distribution of $\eta_{[2,N]}$ converges to some $\al_f\in\mathcal{M}_1(E)$. Note that, under this limit, the corresponding final magnetization 
on the occupied sites 
$\frac{1}{\vert S(\eta_{[2,N]})\vert}\sum_{i\in S(\eta_{[2,N]})}\eta_i$ converges to $\frac{\al_f(1)-\al_f(-1)}{\al_f(1)+\al_f(-1)}$. 

%
Let $\tilde\eta$ be a random variable with mean $\a_f^*$ and $\tilde{\beta}= \frac{\beta}{2}\al_f(\{-1,1\})$.  Together with the Hubbard-Stratonovich analysis which was carried out in detail in \cite{kuelske-le07} this implies that if the function 
\begin{align*}
\phi_{\tilde\beta,t,\al_f^*}(m)= \frac{m^2}{2}- \frac{1}{\tilde{\beta}} \E_{\al_f^*}\left(\log \cosh\left(\tilde\beta\left(m+\frac{h_t}{\tilde\beta}\tilde{\eta}\right)\right)\right)
\end{align*}
has a unique minimizer $m^*$ then under $\mu_{\beta, t, N}[\eta _{N}]$ the empirical magnetization $\frac{1}{\vert S(\eta_N)\vert}\sum_{i\in S(\eta_n)} \omega_i$ converges to this minimizer $m^*$. As a consequence, we obtain the following lemma. 
\begin{lemma}\label{lem: lim first layer}
	Let $\beta>0$ and $t>0$. Assume that $\phi_{\tilde\beta,t,\al_f^*}$ has a unique global minimizer $m^*$. Then it follows that
	\begin{align*}
	\mu_{\beta,t,N}[\eta_{[2,N]}]\Bigl(\phi_{1,N}\bigl(\eta_1, \frac{1}{N-1} \sum_{2\leq j \leq N} \omega_j\bigr) \bigr) \rightarrow \tilde \phi(\eta_1,\al_f(\{-1,1\})m^*)
	\end{align*}
	where $\tilde\phi_1(\eta_1,m) = \al(1)e^{-\beta}(\eta_1e^{-2t}\sinh(\frac{\beta}{2}m)+\cosh(\frac{\beta}{2}m))$ for $\vert \eta_1\vert=1$ and $\tilde{\phi}(0,m)= \al(0)$.
\end{lemma}
\begin{proof}
	First we prove that $\sup_{m\in (0,1) } \vert \tilde{\phi}_{1,N}(\eta_1,m) - \tilde{\phi}_{1}(\eta_1,m) \vert \rightarrow 0$. 
	It is clear  
	that $\tilde{\phi}_{1,n}(\eta_1,m)\rightarrow \tilde{\phi}_{1}(\eta_1,m)$ point-wise.
		Note that we only have to check the uniform convergence for $\vert\eta_1\vert = 1$, and in this case we have 
	\begin{align*}
	\vert \tilde{\phi}_{1,N}(\eta_1,m) - \tilde{\phi}_{1}(\eta_1,m) \vert=\Bigl\vert& \sum_{\omega_1\in\{-1,1\}} e^{\beta (-\frac{N-1}{N}+\frac{N-1}{2N}
	m \omega_1)}\al(\omega_1)\tilde{p}_t(\omega_1,\eta_1) \\
	& - e^{\beta (-1+\frac{1}{2}\omega_1m)}\al(\omega_1)\tilde{p}_t(\omega_1,\eta_1)\Bigr\vert.
	\end{align*}
	Since $\tilde{p}_t$ is positive we can lift it into the exponential. Using the  local Lipschitz property of the exponential function and point-wise convergence there exists some positive $K$ such that for large $N$ it follows that the above difference is bounded by $\sum_{\omega_1} K\beta(\vert 1- \frac{N-1}{N}\vert+\vert m \vert\vert\frac{N-1}{2N} 
	\omega_1-\frac{1}{2}\omega_1\vert)$. 
	The boundedness of $m$  implies the uniform convergence. \\
	The rest of the proof is to show that  $\frac{1}{N-1} \sum_{2\leq j\leq N} \omega_j$ converges against $m^*$ under $\mu_{\beta,t,N}[\eta_{[2,N]}]$. We have
	\begin{align*}
	&\vert	\mu_{\beta,t,N}[\eta_{[2,N]}](\tilde\phi_{1,N}(\eta_1, \frac{1}{N-1} \sum_{2\leq j \leq N} \omega_j) ) - \tilde \phi_1(\eta_1,\a_f(\{-1,1\})m^*) \vert\\
	&\leq\vert \mu_{\beta,t,N}[\eta_{[2,N]}](\tilde\phi_{1,N}(\eta_1, \frac{1}{N-1} \sum_{2\leq j \leq N} \omega_j) -\tilde\phi_1(\eta_1, \frac{1}{N-1} \sum_{2\leq j \leq N} \omega_j))\vert\\
	&+\vert \mu_{\beta,t,N}[\eta_{[2,N]}](\tilde\phi_{1}(\eta_1, \frac{1}{N-1} \sum_{2\leq j \leq N} \omega_j) - \tilde \phi_1(\eta_1,\a_f(\{-1,1\})m^*)\vert
	\end{align*}
	The first summand converges against $0$ by the proven uniform convergence. For the second one this follows by \cite{kuelske-le07} and the fact that $\mu_{\beta,t,N}[\eta_{[2,N]}]$ is a Curie-Weiss model on $S(\eta_{[2,N]})$ and $\frac{1}{N-1} \sum_{2\leq j\leq N} \omega_j =\frac{S(\eta_{[2,N]})}{N-1} \frac{1}{S(\eta_{[2,N]})} \sum_{j \in S(\eta_{[2,N]}) } \omega_j $ under $\mu_{\beta,t,N}[\eta_{[2,N]}]$.
\end{proof}
The next lemma is the last ingredient to prove Theorem \ref{thm: time evolved}.

\begin{lemma}\label{lem: limit exists}
	Let $\beta>0$, $t>0$, $\al,\al_f\in\mathcal{M}_1(E)$ with $\al(+)=\al(-)>0$ and let $(\eta_{[2,N]})_{N\geq 2}$ a sequence with $\lim_{N\rightarrow \infty} \sum_{i=2}^N \delta_{\eta_i}= \al_f$. Furthermore assume that the function $\phi_{\tilde \beta,t,\al_f^*}$ has a unique global minimizer $m^*:=m_{\tilde\beta,t,\al_f^*}^*$ at the effective inverse temperature 
$\tilde{\beta}= \frac{\beta}{2}\al_f(\{-1,1\})$. 	
		Then it follows that $\lim_{N\rightarrow\infty}	\mu_{\beta, t, N}(\eta_1|\eta_{[2,N]}) =\g_{\b,\al,t}(\eta_1 | \al_f)$ exists 
		and is independent of the choice of the sequence, with limiting kernel given by  
	\begin{align*}
	\g_{\b,\al,t}(\eta_1 | \al_f)
	=\frac{\al(0)\mathds{1}_{\eta_1=0}+\al(1)e^{-\beta}(\eta_1e^{-2t}\sinh(\frac{\beta\al_f(\{-1,1\})m^*}{2})+\cosh(\frac{\beta\al_f(\{-1,1\})m^*}{2}))\mathds{1}_{\vert \eta_1\vert=1}}{\al(0)+	2 \al(1) e^{-\beta} \cosh(\frac{\beta\al_f(\{-1,1\})m^*}{2})} 
	\end{align*}
\end{lemma}

Note that, while the set of bad empirical measures does not depend on the value of $\a(0)$, the form 
of the specification kernel does depend on the value of $\a(0)$, wherever it is well-defined.

\begin{proof}
	By \eqref{equi: meas rep } we have a representation of $\mu_{\beta, t, N}(\eta_1|\eta_{[2,N]})$ in terms of the first-layer model. With  the assumption of this theorem it follows by Lemma \ref{lem: lim first layer} that the first-layer model has a limit. The particular form of the specification kernel is given by the function $\tilde\phi_1$ defined in Lemma \ref{lem: lim first layer}.
\end{proof}

\begin{proof}[Proof of Theorem \ref{thm: time evolved} ] 
	With Lemma \ref{lem: limit exists} the existence of the limit of the 
	conditional probability is connected to the unique minimizer of $\phi_{\tilde\beta,t,\al_f^*}$ Luckily the issue of the location of the regions of uniqueness is completely solved by \cite{kuelske-le07}. For a given $\beta$ and $\al_f$ we use their results with  $\tilde{\beta}= \frac{\beta}{2}\al_f(\{-1,1\})$ and magnetization $\al_f^*= \frac{\al_f(1)-\al_f(-1)}{\al_f(\{-1,1\})}$. This leads to the regimes of Gibbsianness for the Curie-Weiss Widom-Rowlinson model.
\end{proof}
\subsection{Time-evolved antiferromagnetic model}
For the antiferromagnetic model there exist no bad empirical measures. In order to see this note that Lemma \ref{lem: lim first layer} is still true for $\beta<0$ and all rewriting of the model does not depend on the sign of $\beta$. Furthermore the function $\phi_{\tilde\beta,t,\al_f^*}$ is strictly convex. Hence for all $\beta<0, t>0$ and $\al_f\in\mathcal{M}_1(E)$ there exists a unique minimizer of $\phi_{\tilde\beta,t,\al_f^*}$ and therefore no bad empirical measures exist by Lemma \ref{lem: limit exists}.
\subsection{Atypicality of bad empirical measures}\label{sec: atypi}

We obtain the minimizers $\nu_t\in\mathcal{M}_1(E)$    of the dynamic rate function 
from the minimizers $\nu_0\in\mathcal{M}_1(E)$ of the static rate function, via the relation
\begin{align}\label{equi: time evo mini}
\nu_t(1)-\nu_t(-1)= e^{-2t}(\nu_0(1)-\nu_0(-1)), 
\end{align} 
together with $\nu_t(0)= \nu_0(0)$, since the hole density does not change over time. 

For $\beta$ below $2$ there are no bad empirical measures, so fix $\beta>2$. 
We split our analysis into two parts. First we consider the symmetric bad empirical measures. In the second part consider 
only the asymmetric ones. 

Note that a symmetric minimizer of the static rate function remains a minimizer of the dynamic rate function, for any time.  
Symmetric minimizers $\nu_{\beta,q}$ of the static model only exist if $\beta\leq\beta_c(q)$ and are then given by the solution of equation \eqref{equ: mean-field} where we defined $q= \frac{\al(0)}{\al(1)}$. Equivalently, symmetric minimizers only exist if $q\in[\frac{\beta-2}{e},\infty)$. 
Furthermore $\nu_{\beta,q}(0)$ is decreasing with decreasing $q$ which implies that  $\nu_{\beta,q}(0)$ is minimal if $q=q_m:= \frac{\beta-2}{e}$. For  $q_m$ we have $\beta=\beta_c(q_m)$ and at this value of $\beta$ we know that $\nu_{\beta,q_m}(0)= 1- \frac{2}{\beta}$ by the proof of Lemma \ref{lem: matrix inde}. This implies that all symmetric bad empirical measures are atypical, for the following reason. 
A necessary condition such that $\al_f$ could be a (symmetric) bad empirical measure  is that the effective inverse 
temperature on the occupied sites $\tilde\b$ is bigger than $1$. But this is equivalent to $\al_f(0)<1-\frac{2}{\beta}$. 

Next we discuss the asymmetric bad empirical measures, assuming 
$\b>3$, 
using the parametrizations of Theorem \ref{closedsolution}.    
To get the curve of asymmetric minimizers of the time-evolved rate function parametrized by $m$ for arbitrary
$\al(0)$ one can rewrite equation \eqref{equi: asym beta} as
\begin{align*}
\al_0(m,\beta) = \frac{\beta-\frac{2}{m}I'(m)}{\beta+\frac{2}{m} I'(m)(-1+e^{  \frac{1}{m}I'(m)-mI'(m)+I(m)})}.
\end{align*}
With the formula of the particle density $x(m,\al(0))$ the curve of asymmetric minimizers of the time-evolved 
rate function of the symmetric Widom-Rowlinson model, after time $t$, is given by 
\begin{align*}
\mathcal{M}_{t}:=\Bigl\{\bigl(x(m,\al_0(m,\beta))\frac{1+m e^{-2t} }{2}, 
x(m,\al_0(m,\beta))\frac{1-m e^{-2t} }{2}
, 1-x(m,\al_0(m,\beta))\bigr)\,:\,{m\in (-m_\beta,m_\beta)} \Bigr\}
\end{align*}
where $m_\beta:=\max \{ m\in(0,1) \,:\, \al_0(m,\beta)>0\}$. \\
By \cite{kuelske-le07} the set of bad empirical measures at fixed $\beta$ and $t<t_3$ is contained in the set $A_{bad}$ whose boundary is given by $4$ curves $C_1,C_2,C_3,C_4$, see the black lines in fig. \ref{fig: atypic_smaller}.  For more details, see Proposition 4.4. where functions $\alpha_{12},\beta_{12}$ describing the relevant 
bifurcation set (which is sheltering the Maxwell lines which relate to the actual bad empirical measures) are introduced.  
Later only the curve $C_1$ will be of interest.  We carry out the map back to the simplex for the Widom-Rowlinson model, 
taking into account effective 
temperature as it relates to repulsion strength and occupation density which gives us 
\begin{align*}
C_1=\{\frac{\beta_{12}(M,h_t)}{\beta}(1+\alpha_{12}(M,h_t)),\frac{\beta_{12}(M,h_t)}{\beta}(1-\alpha_{12}(M,h_t)),
1- 2\frac{\beta_{12}(M,h_t)}{\beta})\,:\, M\in (M_l,M_u) \}
\end{align*}
where $M_u(t):= \argmax_{M>0} \al_{12}(M,h_t)$, $M_l= \inf\{M>0 \,:\, \beta_{12}(M,h_t)=\beta\}$ and we set $\inf\emptyset=0$. The curve
 $C_2$ is identical to the curve $C_1$ mirrored at the  $\pm$-symmetry axis of the simplex.  The curve 
 $C_3$ connects the upper endpoints of $C_1$ and $C_2$. The curve $C_4$ is a line-segment 
in the lower face of the simplex connecting the lower endpoints of the two curves; it is only present in the region 
of two disconnected curves, before the merging to the $Y$-shaped set of bad empirical measures has taken place. 
 
To prove the atypicality of the asymmetric parts of the bad empirical measures in the sense of Definition \ref{asq} it is enough 
to show  that for every $\beta$ and $t$ the intersection of $C_1,C_2$ and $\mathcal{M}_{t}$ is empty. 
This is clear, since we have a concentration of the typical empirical measures for static model, and hence 
also for the dynamic model at any fixed time, which is exponentially fast in the system size.  
 By symmetry we need only to focus on $C_1$ and 
the left arm of $\mathcal{M}_{t}$. 
Hence, the necessary conditions for an intersection are the two equations 
$x(m,\al_0(\beta,m)) = 2 \frac{\beta_{12}(M,h_t)}{\beta}$ and $x(m,\al_0(\beta,m))\frac{m e^{-2t}+1 }{2}= \frac{\beta_{12}(M,h_t)}{\beta}(1+\alpha_{12}(M,h_t))$ for some $m$ and $M$. Combining both equations yields
$m= e^{2t}\alpha_{12}(M,h_t)$, 
and putting this into the first equation gives
$\frac{2}{\beta}\beta_{12}(M,h_t) =x(e^{2t}\alpha_{12}(M,h_t), \alpha_0(e^{2t}\alpha_{12}(M,h_t),\beta))$
The r.h.s. simplifies to $\frac{2I'(e^{2t}\alpha_{12}(M,h_t))}{\beta e^{2t}\alpha_{12}(M,h_t)}$ which implies that the last equation does not depend on $\beta$. By the analysis of Section \ref{sec: proofs static} the function $\frac{I'(m)}{m}$ is monotonically increasing for $m>0$ and by \cite{kuelske-le07} it is known that $\alpha_{12}$ is monotonically increasing from $0$ to $M_u(t)$. Also it is known that $\beta_{12}$ is monotonically decreasing from $0$ to $M_u(t)$.  Therefore it suffices to consider the case of  $M=M_u(t)$. In this way we can reduce the proof of atypicality of 
non-symmetric bad empirical measures for all parameters for which they possibly occur, to showing the following inequality 
for a  function of one variable (namely time $t$) on a compact interval 
\begin{align}\label{equi: strange function}
[0,t_3]\ni t\mapsto \frac{I'(e^{2t}\alpha_{12}(M_u(t),h_t))}{e^{2t}\beta_{12}(M_u(t),h_t)\alpha_{12}(M_u(t),h_t)}< 1
\end{align}
Numerics shows that the l.h.s. is increasing, and as $\lim_{x \downarrow 0}I'(x)/x=1$, the sup is achieved at the right endpoint 
with value $\beta^{-1}_{12}(M_u(t_3),h_{t_3})=\frac{2}{3}$. 
\begin{figure}[!ht]\label{fig: atypic_smaller}\centering
	{\includegraphics[trim = 0mm 40mm 0mm 0mm  , clip,width=0.4\textwidth]{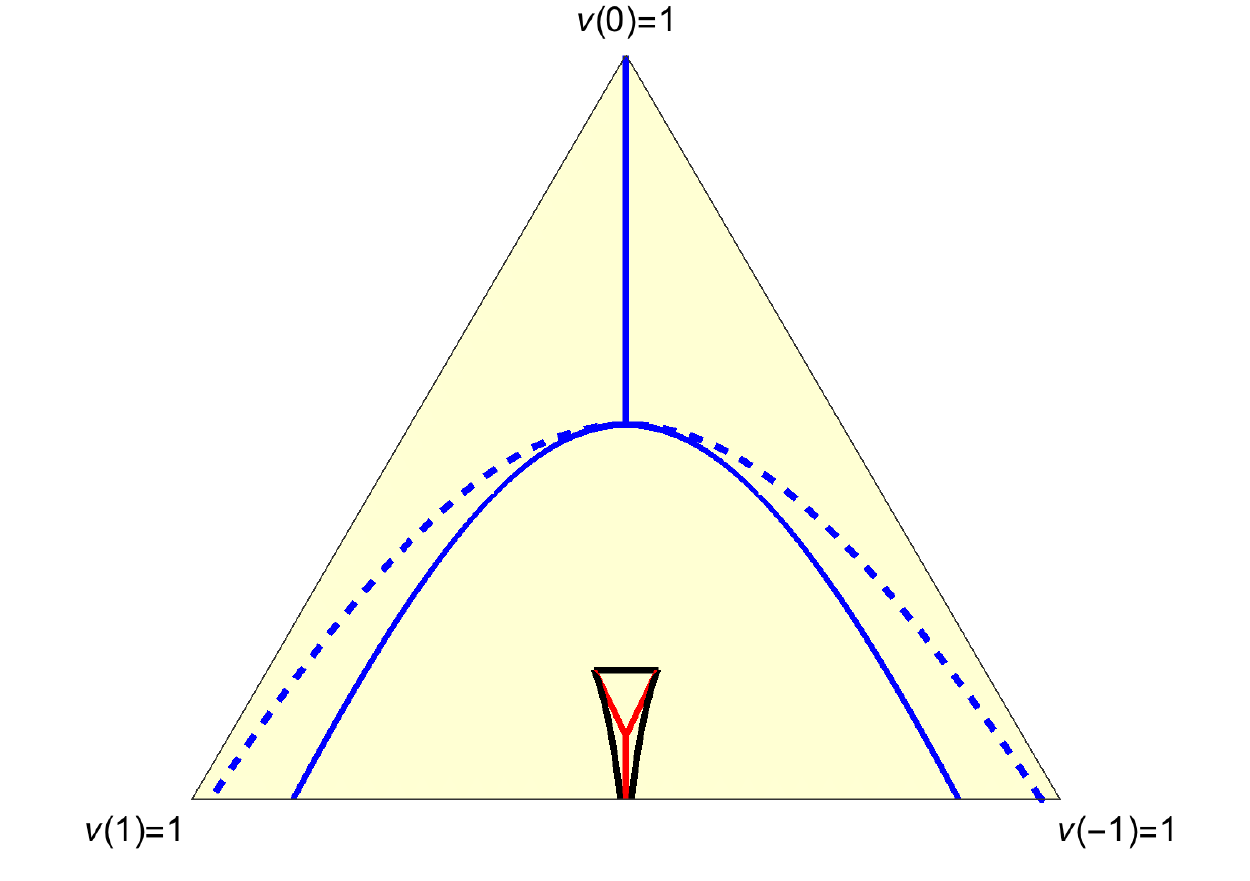}}
	\caption{Bifurcation set (black) bounding the set of bad
	empirical measures (red)}
\end{figure}
\section*{Acknowledgements}
We thank Richard Kraaij for pointing out the connection between semi-concavity and bad magnetizations for the Curie-Weiss model.\\
Sascha Kissel has been supported by
the German Research Foundation (DFG) via Research Training Group RTG 2131 \textit{High dimensional
	Phenomena in Probability - Fluctuations and Discontinuity}.

 \bibliographystyle{plain}
 \bibliography{mean-field-sc}

\end{document}